\documentclass{article}

\usepackage[preprint]{neurips_2025}

\usepackage[utf8]{inputenc} % allow utf-8 input
\usepackage[T1]{fontenc}    % use 8-bit T1 fonts
\usepackage{xcolor}         % colors
\definecolor{neuripsblue}{rgb}{0.21,0.49,0.74}
\usepackage[pagebackref,breaklinks,colorlinks,allcolors=neuripsblue]{hyperref}
\usepackage{url}            % simple URL typesetting
\usepackage{booktabs}       % professional-quality tables
\usepackage{amsfonts}       % blackboard math symbols
\usepackage{nicefrac}       % compact symbols for 1/2, etc.
\usepackage{microtype}      % microtypography
\usepackage{dsfont}
\usepackage{amsmath,amssymb,amsthm}
\usepackage{algorithm}
\usepackage[noend]{algpseudocode}
\usepackage{graphicx}
\usepackage{multirow}
\usepackage{subcaption}
\usepackage{lipsum}
\usepackage{mathtools}

\newtheorem{theorem}{Theorem}[section]
\newtheorem{lemma}[theorem]{Lemma}
\newtheorem{proposition}[theorem]{Proposition}

\newtheorem{assumption}[theorem]{Assumption}

\allowdisplaybreaks

\algnewcommand{\LineComment}[1]{\State \textbf{//} \textit{#1}}
\usepackage{algpseudocode}
\algnewcommand{\Break}{\State \textbf{break}}

\newcommand{\R}{\mathbb{R}}

\newcommand{\ind}{\mathds{1}}

\DeclareMathOperator*{\argmin}{arg\,min}

\renewcommand{\H}{\mathcal{H}}
\newcommand{\B}{\mathcal{B}}
\newcommand{\D}{\mathrm{D}}

\newcommand{\RelErr}{\mathrm{RelErr}}
\newcommand{\dif}{\mathrm{d}}

\newcommand\inner[1]{\left\langle#1\right\rangle}
\newcommand\norm[1]{\left\lVert#1\right\rVert}
\newcommand{\render}{{R}}
\newcommand{\near}{\mathrm{near}}
\newcommand{\far}{\mathrm{far}}

\title{
    Functional Gradient Descent \\ with Adaptive Representations
}

\author{%
   Daniel Csillag$^{1}$\thanks{Corresponding author: Daniel Csillag (\texttt{daniel.csillag@fgv.br})} \quad Rodrigo Schuller$^2$ \quad Pedro Dall'Antonia$^1$ \\[0.2cm]
   \textbf{Leonidas Guibas}$^3$ \quad \textbf{Luiz Velho}$^{2}$\quad \textbf{Tiago Novello}$^{2,3}$ \\[0.2cm]
   $^{1}${\small FGV EMAp} \quad $^{2}${\small IMPA} \quad $^{3}${\small Stanford University} %
}

\begin{document}

\maketitle

\begin{abstract}
    Functional optimization problems are typically solved by optimizing the parameters of a fixed representation, such as a neural network,
    resulting in highly nonconvex losses that complicate both training and theoretical analysis.
    An interesting alternative is functional gradient descent (FGD), that is, gradient descent directly in function space, which benefits from strong convergence results and admits a clean theory.
    However, FGD is difficult to implement in practice because functional gradients are infinite-dimensional, and thus cannot be fully computed nor stored in memory.
    Existing implementations therefore rely on fixed approximations, which introduce approximation error.
    We propose a new, theoretically-grounded FGD algorithm that adapts the representation of the functional gradients over the course of optimization.
    By explicitly incorporating this approximation into the analysis, we establish convergence to a stationary point (for smooth losses) and to a global minimizer (under smoothness + a Polyak-Łojasiewicz-type condition) regardless of our approximations.
    To the best of our knowledge, this is the first implementable FGD method with such guarantees in a general setting.
    We demonstrate the effectiveness of our method on regression, numerical solution of PDEs, and modern computer vision. Across settings, our method consistently outperforms both FGD with fixed approximations and neural network baselines in efficiency and accuracy.
\end{abstract}

\section{Introduction}

Functional optimization problems are core to many areas, including machine learning, statistics, signal processing, scientific computing, and more.
In these settings, the objective is to optimize a \emph{function} over a function space to minimize some loss (e.g., mean squared error or cross-entropy).

To solve such problems, it is common to represent the target function as a neural network and optimize the loss over the network's weights through gradient descent.
While flexible, this introduces some challenges:
(i) the resulting optimization problem is highly nonconvex, making training costly;
(ii) the representation is typically fixed \textit{a priori}, limiting its ability to adapt during optimization; and
(iii) the nonconvex nature of the parametrized loss significantly complicates theoretical analyses of convergence and training dynamics.

We instead consider gradient descent directly in function space --- \emph{functional gradient descent} (FGD) --- whose dynamics are generally simpler than their parameterized counterparts and admit strong convergence guarantees.
Formally, consider the optimization problem
\begin{equation}\label{eq:optimization-problem}
    \argmin_{f \in \H} L(f),
\end{equation}
where $\H$ is a Hilbert space of functions (e.g., $L^2$, Sobolev spaces, or an RKHS\footnote{Reproducing Kernel Hilbert Space}).
Under suitable differentiability conditions, the (non-parametric) loss admits a well-defined functional gradient $\nabla L(h) \in \H$, allowing us to write the FGD updates
\begin{equation}\tag{Ideal FGD}\label{eq:ideal-fgd}
    f_{t+1} = f_t - \eta \nabla L(f_t).
\end{equation}
From a theoretical point of view, FGD is well known to benefit from good convergence guarantees (cf. e.g. \citep{book-bauschke,book-opt-hilbert}).
However, it is nontrivial to implement in practice.
Since the functional gradients reside in the infinite-dimensional space $\H$, they generally cannot be computed exactly nor stored in memory.
As a result, to actually implement FGD we must \emph{approximate} the gradients $\nabla L(f_t)$ via some finite-dimensional representation.
Yet using a fixed representation introduces approximation error, preventing convergence to a global minimizer.

In this paper, we propose a new theoretically-grounded functional gradient descent algorithm that adapts the representation of the functional gradients over the course of optimization.
We show that by adaptively refining the gradient approximations so as to ensure a certain bound on a relative error condition, we guarantee sufficient descent and thus proper convergence.
To the best of our knowledge, \textbf{ours is the first implementable FGD method that guarantees convergence to a global minimizer.}
The resulting procedure strictly outperforms functional gradient descent with any fixed representation, and simultaneously outperforms neural networks in both training speed and quality; see Figure~\ref{fig:teaser}.

\begin{figure}[t]
    \centering
    \includegraphics[width=\columnwidth]{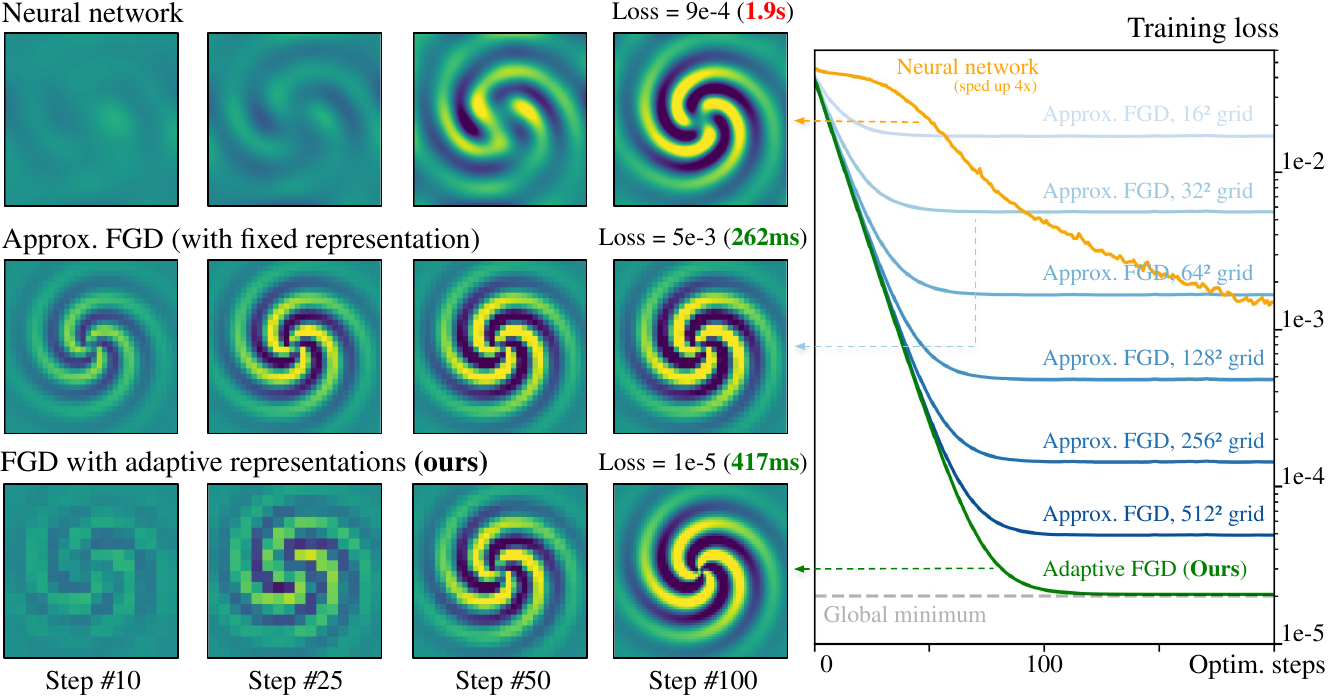}
\caption{\textbf{Our method on a toy MSE minimization task.} 
We compare adaptive FGD (\textcolor[rgb]{0.13,0.55,0.13}{ours}) with (i) \textcolor[rgb]{0.12,0.47,0.71}{approximate FGD} using fixed representations, which converge quickly but plateau at a resolution-dependent error, and (ii) a \textcolor[rgb]{1.0,0.5,0.06}{neural network} baseline trained with Adam, 
which converges more slowly. 
In contrast, our method starts from a coarse representation and automatically refines it, 
achieving faster convergence and reaching the global minimum.}
    \label{fig:teaser}
\end{figure}

Our \textbf{main contributions} are:

\begin{itemize}
    \item We establish sufficient conditions for the proper convergence of approximate FGD. Notably, our analysis covers the general setting where true gradients lie in a Hilbert space but are approximated within a broader Banach space, enabling important functional optimization tasks. We prove that maintaining a specific relative error bound guarantees convergence to a stationary point for smooth losses, and to a global minimizer under an additional Polyak-Łojasiewicz-type condition.
    \item We develop novel gradient approximation strategies that provably satisfy our derived relative error condition, using only computable quantities. This yields the first fully implementable FGD algorithms that retain strict, end-to-end convergence guarantees.
    \item We evaluate our method on three representative tasks across regression in an RKHS, numerical PDE solvers, and modern computer vision. On all settings, our adaptive approach consistently surpasses FGD using any fixed representation, while substantially outperforming neural network baselines in both training efficiency and solution quality.
\end{itemize}

\paragraph{Related work}
FGD is well-established as a theoretical procedure; see, e.g. \citep{book-opt-hilbert} for an overview of the convergence theory for exact FGD.
In specific cases, exact FGD is actually feasible to implement; for example, in the specific case of standard regression losses over an RKHS, the gradients are finite-dimensional and can therefore be exactly represented. This is akin to the classic representer theorems for kernel regression \citep{representer-theorems-wahba,representer-theorems-scholkopf}.
However, such approaches rely on specific finite-sample structures, do not generalize beyond simple regression settings, and typically scale poorly with sample size.
Alternatively, several works approximate functional gradients \citep{sip-functional,npiv,fgd-random-gradientfree,functional-bilevel,wasserstein-barycenter-fgd},
but introduce irreducible approximation error and often do not incorporate this approximation into their theoretical analysis.
Our work addresses this by (i) providing a general theory for approximate FGD, and (ii) identifying sufficient conditions that ensure convergence to proper minimizers sans approximation error.

Our work is also related to first-order optimization with inexact gradients \citep{inexact-gd-1,inexact-gd-2,inexact-gd-3}, but differs in that we operate in the infinite-dimensional setting, where many previous methods break down.
In particular, most approximation schemes (e.g. \citep{finite-diff-gd-1,book-opt-hilbert,finite-diff-gd-3}) do not extend to infinite dimensions.
Additionally, we consider a more general setting in which gradients lie in a Hilbert space but are approximated in a broader Banach space, beyond the scope of existing theory.

\section{Background: function spaces and functional gradients}\label{sec:background}

We briefly recall function spaces and differentiability in infinite dimensions.
A Hilbert space $\H$ is a complete inner product space, endowed with $\langle \cdot, \cdot \rangle_\H$ and induced norm $\lVert h \rVert_\H = \sqrt{\langle h, h \rangle_\H}$.
Similarly, a Banach space $\B$ is a complete normed space with norm $\lVert \cdot \rVert_\B$, but generally does not admit an inner product.
Every Banach space $\B$ admits a dual space $\B^*$, consisting of continuous linear functionals on $\B$, which is itself a Banach space under the dual norm $\lVert u \rVert_{\B^*} = \sup_{\lVert b \rVert_\B \leq 1} \lvert u(b) \rvert$.

Let $L : \H \to \R$ be a functional.
The directional derivative of $L$ at $f$ in the direction $\vec{f}$ is defined as 
\[ \D L(f; \vec{f}) := \lim_{\delta \searrow 0} \frac{L(f + \delta \vec{f}) - L(f)}{\delta}. \]
If $\D L$ exists for all $f, \vec{f}$ and is linear over $\vec{f}$, $L$ is Gâteaux differentiable;
additionally, if $\D L$ is continuous in $f$, then $L$ is Fréchet differentiable.\footnote{This is equivalent to the usual definition of Fréchet derivatives as bounded linear operators, \citep{book-intro-to-nonlinear-analysis}.}

When $L$ is Fréchet differentiable over a Hilbert space $\H$, for every $f \in \H$ the map $\D L(f; \cdot)$ (i.e., the functional $\vec{f} \mapsto \D L(f; \vec{f})$) is a continuous linear functional. Therefore, by the Riesz representation theorem, there exists a unique element $\nabla L(f) \in \H$ such that
\begin{align}\label{e-functional-gradient}
    \D L(f; \vec{f}) = \langle \nabla L(f), \vec{f} \rangle_\H \qquad\quad \text{for all } \vec{f} \in \H.
\end{align}
This element $\nabla L(f)$ is the \emph{functional gradient}, defining the gradient operator $\nabla L : \H \to \H$.
When $\H = \R^d$ with the standard inner product, this recovers the usual notion of gradients.

Finally, while Hilbert spaces provide a convenient structure for defining gradients, they can impose significant constraints on their elements.
For instance, the empirical MSE functional
\[ L_\mathrm{MSE}(f) = \frac{1}{n} \sum_{i=1}^n \bigl( f(x_i) - y_i \bigr)^2\]
is not well-defined on $L^2(\R)$ (space of square-integrable functions), since functions in $L^2(\R)$ cannot be evaluated on a set of measure zero.
In contrast, it is well-defined in RKHSs, which enforce additional regularity but are more restrictive than $L^2(\R)$.

Nevertheless, these issues arise due to the requirement of a proper inner product. Once we move from Hilbert to Banach spaces, we can use flexible function spaces such as $L^\infty$ or the space of continuous functions, without caveats.

\section{Method}

\begin{algorithm}[t]
    \caption{Functional gradient descent with adaptive representations}\label{alg:main-algorithm}
    \begin{algorithmic}
        \Require loss $L : \H \to \R$, learning rate $\eta$, initial condition $f_0$, tolerance $\epsilon$
        \State Choose an initial representation
        \For{$t = 0, 1, \ldots, T-1$}
        \vspace{0.1cm}
            \LineComment{Approximate the functional gradient}
            \While{true}
                \State $g_t \hspace{2pt}\gets \text{approximate $\nabla L(f_t)$ using the current representation}$
                \State $U_t \gets \text{tight upper bound on $\lVert g_t - \nabla L(f_t) \rVert_\B$}$
             \State \textbf{if} $(1+\epsilon) U_t < \epsilon \lVert g_t \rVert_\B$ \textbf{then break}
                \State Refine the current representation
            \EndWhile
            \vspace{0.2cm}
            \LineComment{Do the gradient descent step}
            \State $f_{t+1} \gets f_t - \eta g_t$
        \EndFor
        \State \Return $f_T$
    \end{algorithmic}
\end{algorithm}

Let $\H$ be a Hilbert space contained in a Banach space $\B$,
and recall the optimization problem \eqref{eq:optimization-problem}
\[ \argmin_{f \in \H} L(f), \]
where $L : \H \to \mathbb{R}$ is taken to be a Fréchet differentiable functional on $\H$.
Since $\H$ is a Hilbert space, the derivative of $L$ admits a gradient operator $\nabla L : \H \to \H$ (cf. Section~\ref{sec:background}).
An ideal functional gradient descent (FGD) scheme would, therefore, perform the steps
\[ f^\mathrm{ideal}_{t+1} = f^\mathrm{ideal}_t - \eta \nabla L(f^\mathrm{ideal}_t), \]
which remain strictly within $\H$ and whose convergence follows by the standard proofs of convergence of FGD.
However, since this is intractable to implement, we instead replace the gradient by an \emph{approximate gradient} $g_t$, yielding updates
\begin{equation}\label{eq:our-updates}
    f_{t+1} = f_t - \eta g_t,
\end{equation}
where $g_t \in \B$ (and thus the iterates $f_t$ also reside in $\B$).
We will show that as long as $g_t$ approximates $\nabla L(f_t)$ reasonably well (in a sense we will make precise), we still guarantee descent and proper convergence.
Subsequently, we shall study strategies to guarantee that the required approximation conditions hold in practice by adaptively refining the gradient representations.

\subsection{Convergence}\label{sec:method-convergence}

In order to analyze the convergence of the updates in \autoref{eq:our-updates}, we must introduce some assumptions.
First, note that loss and gradient operator are defined over the Hilbert space $\H$, but our updates will reside in the larger Banach space $\B$;
therefore, in order to keep things well-defined, our first assumption is that both the loss $L$ and the gradient operator $\nabla L$ have extensions from $\H$ onto $\B$:
\begin{assumption}[Extension]\label{assumpt:continuous-extension}
    There exists a Fréchet differentiable functional $\overline{L} : \B \to \R$ and an operator $\overline{\nabla L} : \B \to \B$ such that, for any $f \in \H$,
    \[ \overline{L}(f) = L(f) \qquad\text{and}\qquad \overline{\nabla L}(f) = \nabla L(f). \]
\end{assumption}
For convenience, from here on we will slightly abuse notation and ``shadow'' $L$ and $\nabla L$ by their extended counterparts; i.e., for some $f \in \B$ we shall write $L(f)$ instead of $\overline{L}(f)$, and similarly we shall write $\nabla L(f)$ instead of $\overline{\nabla L}(f)$.

We will additionally rely on the standard assumptions used for the analysis of gradient descent, albeit over the larger Banach space $\B$ rather than the usual Hilbert $\H$.
In particular, we will assume that the loss is $K$-smooth:
\begin{assumption}[$K$-smoothness]\label{assumpt:smoothness}
    For any $f, f' \in \B$, it holds that
    \[ L(f) \leq L(f') + \D L(f'; f - f') + \frac{K}{2} \lVert f - f' \rVert_\B^2. \]
\end{assumption}
Note that we write the $K$-smoothness bound in terms of the directional derivative $\D L$ rather than an inner product with the gradient, as the Banach space does not generally have an inner product.

We will also require two additional assumptions, which serve to connect the gradients between the Banach space $\B$ and the Hilbert space $\H$:
\begin{assumption}[$\H$ descends in $\B$]\label{assumpt:h-descends-in-b}
    There exists some constant $\alpha>0$ such that, for all $f \in \B$,
    \[ \D L(f; \nabla L(f)) \geq \alpha \lVert \nabla L(f) \rVert_\B^2. \]
\end{assumption}
\begin{assumption}[Gradient compatibility]\label{assumpt:grad-compat}
    There exists a constant $\beta > 0$ such that, for every $f \in \B$,
    \[ \lVert \D L(f; \cdot) \rVert_{\B^*} \leq \beta \lVert \nabla L(f) \rVert_{\B}. \]
\end{assumption}
When $\B = \H$ both of these assumptions are immediately satisfied with $\alpha = \beta = 1$: $\D L(f; \nabla L(f)) = \langle \nabla L(f), \nabla L(f) \rangle_\H = \lVert \nabla L(f) \rVert_\H^2$, and $\lVert \D L(f; \cdot) \rVert_{\B^*} = \lVert \nabla L(f) \rVert_{\B}$.
But these assumptions are more general, and satisfied in some important cases where $\H \neq \B$ (cf. Section~\ref{sec:exp1}).
Intuitively, Assumption~\ref{assumpt:h-descends-in-b} asserts that the direction given by the extended gradient $\nabla L(f)$ is guaranteed to be a direction of sufficient descent per the geometry of $\B$. Assumption~\ref{assumpt:grad-compat}, on the other hand, establishes a sort of local equivalence of norms over the gradients (but not over the whole space).

With these in hand, we can analyze a single step of our method (Equation~\ref{eq:our-updates}) to establish a sufficient descent lemma.
The following result follows from the $K$-smoothness inequality (Assumption~\ref{assumpt:smoothness}), along with our Assumptions~\ref{assumpt:h-descends-in-b} and \ref{assumpt:grad-compat} and the use of triangular inequalities:

\begin{lemma}[Sufficient descent lemma]\label{lemma:sufficient-descent-lemma}
    Let $f_t, f_{t+1}$ be as in \eqref{eq:our-updates}.
    Suppose $L$ and $\nabla L$ admit extensions satisfying $K$-smoothness, gradient compatibility and that $\mathcal{H}$ descends in $\mathcal{B}$ (Assumptions~\ref{assumpt:continuous-extension}-\ref{assumpt:grad-compat}).
    Then, for any learning rate $\eta > 0$ and step $t$ with $\RelErr(g_t, \nabla L(f_t)) \leq 1/2$,
    \begin{align*}
        L(f_{t+1})
        \leq L(f_t) - \eta \left[ \alpha - \frac{K \eta}{2} - \left( \beta + \frac{3}{2} K \eta \right) \frac{\RelErr(g_t, \nabla L(f_t))}{1-\RelErr(g_t, \nabla L(f_t))} \right] \lVert \nabla L(f_t) \rVert_\B^2,
    \end{align*}
    where
    \begin{equation}
        \RelErr(g, \nabla L(f)) := \frac{\lVert g - \nabla L(f) \rVert_\B}{\lVert g \rVert_\B}.
    \end{equation}
\end{lemma}

Note that this bound is tight, in the sense that when $\H = \B$, $g_t = \nabla L(f_t)$, and $\eta = 1/K$, we recover the usual descent lemma $L(f_t - \eta g_t) \leq L(f_t) - \frac{\eta}{2} \lVert \nabla L(f_t) \rVert_\B^2$ (cf. e.g. \citep{optimization-book}).

As long as the term within the brackets is strictly positive, we guarantee that a step of approximate functional gradient descent decreases the loss in a manner proportional to the current gradient value.
By telescoping and rearranging we obtain:

\begin{proposition}[Convergence to a stationary point]\label{prop:convergence-to-stationary-point}
    Under the setting of Lemma~\ref{lemma:sufficient-descent-lemma}, furthermore assume that the loss is lower bounded by $L^\star$,
    and that the relative error is uniformly bounded by a constant $\epsilon < \min \{ 1/2, \alpha/(\alpha+\beta) \}$ throughout iterations:
    \[ \RelErr(g_t, \nabla L(f_t)) \leq \epsilon \qquad\qquad \text{for all iterations } t = 0, \ldots, T-1. \]
    Then, for any learning rate $\eta \leq 2 (\alpha - (\alpha + \beta) \epsilon) / K ( 2\epsilon + 1 )$,
    \[ \min_{t = 0, \ldots, T-1} \lVert \nabla L(f_t) \rVert_\B^2 \leq \frac{L(f_0) - L^\star}{T \eta r}, \]
    where $r>0$ is a constant defined in the proof.
\end{proposition}

Under additional structural assumptions we are able to furthermore establish convergence to the proper global minimizer.
For our purposes, we will work under a Polyak-Łojasiewicz-type assumption, which serves as a generalization of strong convexity:

\begin{assumption}[Polyak-Łojasiewicz condition on $\H$ over $\B$]\label{assumpt:polyak-lojasiewicz}
    The extended Fréchet-differentiable functional $L : \B \to \R$ is $\mu$-Polyak-Łojasiewicz if, for any $f \in \B$,
    \[ L(f) - L^\star \leq \frac{1}{2\mu} \lVert \D L(f; \cdot) \rVert_{\B^*}^2\,,\quad
    \text{where}\quad  L^\star = \inf_{f' \in \H} L(f').\]
\end{assumption}

It then holds:

\begin{proposition}[Global convergence]\label{prop:convergence-to-global-minimizer}
    Under the conditions of Proposition~\ref{prop:convergence-to-stationary-point}, further assume that $L$ is $\mu$-Polyak-Łojasiewicz (Assumption~\ref{assumpt:polyak-lojasiewicz}).
    Then, for any learning rate $\eta < 2 (\alpha - (\alpha + \beta) \epsilon) / K (2\epsilon + 1)$, performing $T$ iterations of approximate functional gradient descent guarantees that
    \begin{align*}
        L(f_T) - L^\star
        \leq [1 - 2 \eta \beta^{-2} \mu r]^T \, \bigl( L(f_0) - L^\star \bigr).
    \end{align*}
\end{proposition}

\subsection{Adaptive representations of functional gradients}\label{sec:method-fitting}

In the previous section we have established that as long as we can ensure an upper bound on the relative error, we guarantee proper convergence.
In this section, we now consider how exactly we can approximate the functional gradient so as to guarantee this will be the case.

Recall that our ultimate objective is to find some gradient approximation $g$ that ensures that the relative error is bounded by some fixed $\epsilon$:
\[ \RelErr(g, \nabla L(f)) = \frac{\lVert g - \nabla L(f) \rVert_\B}{\lVert g \rVert_\B} \leq \epsilon. \]
Note that while the denominator can generally be computed exactly (as it is the norm of our own approximation), the approximation error in the numerator can still be tricky. However, it is often possible to compute a (reasonably tight) upper bound $U(g, \nabla L(f))$ on $\lVert g - \nabla L(f) \rVert_\B$, in the sense that\looseness=-1
\begin{equation}\label{eq:tight-approx-error-bound}
    \begin{aligned}
        \lVert g - \nabla L(f) \rVert_\B &\leq U(g, \nabla L(f)) &\text{for all $g, f \in \B$}.
    \end{aligned}
\end{equation}
With such an upper bound in hand, we can guarantee:
\begin{lemma}\label{lemma:approx-error-limit}
    Given some $f \in \B$ and desired tolerance $\epsilon>0$,
    let $U(\cdot, \cdot)$ be an upper bound on the approximation error (as in \eqref{eq:tight-approx-error-bound}), and
    let $(g_1, g_2, \ldots)$ be a sequence of gradient approximations in $\B$ such that $U(g_n, \nabla L(f)) \to 0$.
    Then there exists some $n$ such that
    \begin{equation}\label{eq:relerr-condition}
        \frac{\RelErr(g_n, \nabla L(f))}{1 - \RelErr(g_n, \nabla L(f))} \leq \epsilon.
    \end{equation}
\end{lemma}

The result in Lemma~\ref{lemma:approx-error-limit} is quite general.
Importantly, it suggests a concrete algorithm: at each FGD step we design a sequence of representations, and loop over this sequence; at each iterate, approximate the gradient using the current representation and compute the upper bound on the approximation error. As soon as Equation~\ref{eq:relerr-condition} is satisfied (which must happen at some point, as per Lemma~\ref{lemma:approx-error-limit}), we use the current approximation.
Algorithm~\ref{alg:main-algorithm} gives the overall procedure.

Connecting with our previous convergence theory, we obtain:

\begin{theorem}\label{thm:main-theorem}
    Let $f_t$ be the iterates produced throughout Algorithm~\ref{alg:main-algorithm}, with $\epsilon \in (0, 1)$. Then:
    \begin{enumerate}
        \item[(i)] \textbf{Relative error bound:}
            For all iterations $t = 1, \ldots, T$,
            \[ \RelErr( g_t, \nabla L(f_t) ) \leq \frac{\epsilon}{1+\epsilon} < 1/2; \]
        \item[(ii)] \textbf{Convergence to a stationary point:}
            Under the conditions of Proposition~\ref{prop:convergence-to-stationary-point},
            \[ \min_{t = 0, \ldots, T-1} \lVert \nabla L(f_t) \rVert_\B^2 \leq \frac{L(f_0) - L^\star}{T \eta r}, \]
            where $r>0$ is a constant defined in the proof; and
        \item[(iii)] \textbf{Convergence to a global minimizer:}
            Under the conditions of Proposition~\ref{prop:convergence-to-global-minimizer},
            \[
                L(f_T) - L^\star \leq [1 - 2 \eta \beta^{-2} \mu r]^T \, \bigl( L(f_0) - L^\star \bigr).
            \]
    \end{enumerate}
\end{theorem}

\section{Experiments}

In this section we empirically assess our method on three representative functional optimization tasks:
regression in an RKHS, solution of partial differential equations and inverse rendering.
Throughout, we focus on comparing our method against (i) neural networks (with appropriate engineering), and (ii) FGD with fixed representations.

All experiments were run on a desktop with an Intel Core i9-14900KF CPU and NVIDIA GeForce RTX 4090 GPU, with 128GB of RAM and 24GB of VRAM, respectively.
That said, the experiments are relatively cheap and should run on weaker hardware.
Code to reproduce our results can be found at \url{https://github.com/dccsillag/experiments-adaptive-fgd}.
In all experiments the learning rates were tuned to obtain the best possible performance of each method.

\subsection{Regression in an RKHS}\label{sec:exp1}

\begin{figure}[t]
    \centering
    \begin{minipage}{.5\textwidth}
        \small
        \begin{tabular}{lcccc}
            \toprule
             & \multicolumn{2}{c}{MSE} & \multicolumn{2}{c}{Cross-Entropy} \\
            \cmidrule(lr){2-3}\cmidrule(lr){4-5}
            Method
                & \shortstack{Test \\ Loss} & \shortstack{Train \\ Time}
                & \shortstack{Test \\ Loss} & \shortstack{Train \\ Time}
                \\
            \midrule
            Neural network
                & 12.84 & 1.02s
                & 0.3325 & 1.05s
                \\
            Approx. FGD (depth 2)
                & 0.1310 & 257ms
                & 0.7167 & 209ms
                \\
            Approx. FGD (depth 4)
                & 0.0900 & 232ms
                & 0.5134 & 234ms
                \\
            Approx. FGD (depth 8)
                & 0.0588 & 274ms
                & 0.4024 & 194ms
                \\
            Approx. FGD (depth 12)
                & 0.0444 & 762ms
                & 0.2916 & 806ms
                \\
            Adaptive FGD \textbf{(ours)}
                & \textbf{0.0378} & 843ms
                & \textbf{0.2434} & 815ms
                \\
            \bottomrule
        \end{tabular}
    \end{minipage}\hfill
    \begin{minipage}{.35\textwidth}
        \vspace{0.25cm}
        \includegraphics[width=\textwidth]{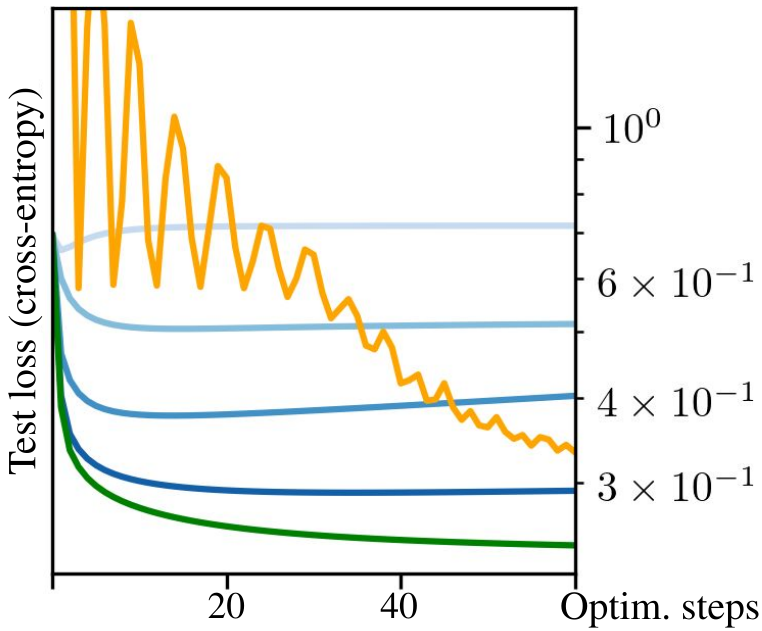}
    \end{minipage}
    \caption{\textbf{Our method for regression in an RKHS.} We compare our method against FGD with a fixed tree representation and a neural network for regression. Our method quickly converges towards the global minimizer with a good test loss, whereas fixed-representation FGD stagnates at a higher loss, and the neural network is slow to converge.}\label{fig:exp1}
\end{figure}

In this section we consider the optimization problem
\begin{equation}
    \argmin_{f \in \H_K} L_\mathrm{REG}(f), \qquad\text{with}\qquad L_\mathrm{REG}(f) := \frac{1}{n} \sum_{i=1}^n \ell\bigl( f(X_i), Y_i \bigr),
\end{equation}
where $\H_K$ is an RKHS with kernel $K$.
Here, $\ell$ denotes a pointwise loss such as the mean squared error ($\ell(\hat{y}, y) = \frac{1}{2} (\hat{y} - y)^2$) or binary cross-entropy ($\ell(\hat{y}, y) = -y \log \hat{y} - (1-y) \log (1-\hat{y})$).
Such regression problems are commonplace in statistical inference procedures, including conditional independence testing \citep{kernel-ci-test}, asymptotic statistics \citep{van-der-vaart}, and more.

The loss $L_\mathrm{REG}$ admits simple functional gradients, is $K$-smooth and Polyak-Łojasiewicz:

\begin{proposition}\label{prop:gradient-kernel}
    The functional $L_\mathrm{REG} : \mathcal{H}_K \to \R$ has gradients given by
    \[ \bigl[ \nabla L_\mathrm{REG}(f) \bigr](x) = \frac{1}{n} \sum_{i=1}^n \ell'(f(X_i), Y_i) K(X_i, x), \]
    where $\ell'$ denotes the derivative of $\ell$ with respect to its first argument.
    Moreover, if $\ell$ is strongly-convex in its first argument, then $L_\mathrm{REG}$ is Polyak-Łojasiewicz, and if $\ell$ is $K$-smooth then so is $L_\mathrm{REG}$.
    \looseness=-1
\end{proposition}

As mentioned in Section~\ref{sec:background}, the space $\H_K$ is generally quite constrained.
For this reason, in this experiment we opt to approximate the gradients in the larger Banach space $L^\infty$ of bounded functions, which can be checked to satisfy our gradient bridging assumptions (Assumptions~\ref{assumpt:h-descends-in-b} and \ref{assumpt:grad-compat}), cf. Proposition~\ref{prop:gradient-bridging-rkhs}.
For our adaptive representations we then choose a decision tree-based approximation scheme inspired by the classic CART algorithm (which reside in $L^\infty$, but generally not in $\H_K$).
By using sufficiently deep trees we are able to approximate functions arbitrarily well, allowing us to use Lemma~\ref{lemma:approx-error-limit} and thus ensure convergence as per Theorem~\ref{thm:main-theorem}.

For this experiment we use the binary classification dataset of \citep{libsvm-data} for the detection of non-coding RNA sequences.
We consider two losses $\ell$: mean squared error and binary cross entropy.
Figure~\ref{fig:exp1} shows the results of our experiment.
We find that our method consistently outperforms a standard MLP neural network as well as FGD with any fixed representation, adapting as necessary to the target function and guaranteeing convergence to the global optimum while achieving a better test loss.
\looseness=-1

\subsection{Solving the wave equation}\label{sec:exp2}

\begin{figure}[t]
    \includegraphics[width=0.95\columnwidth]{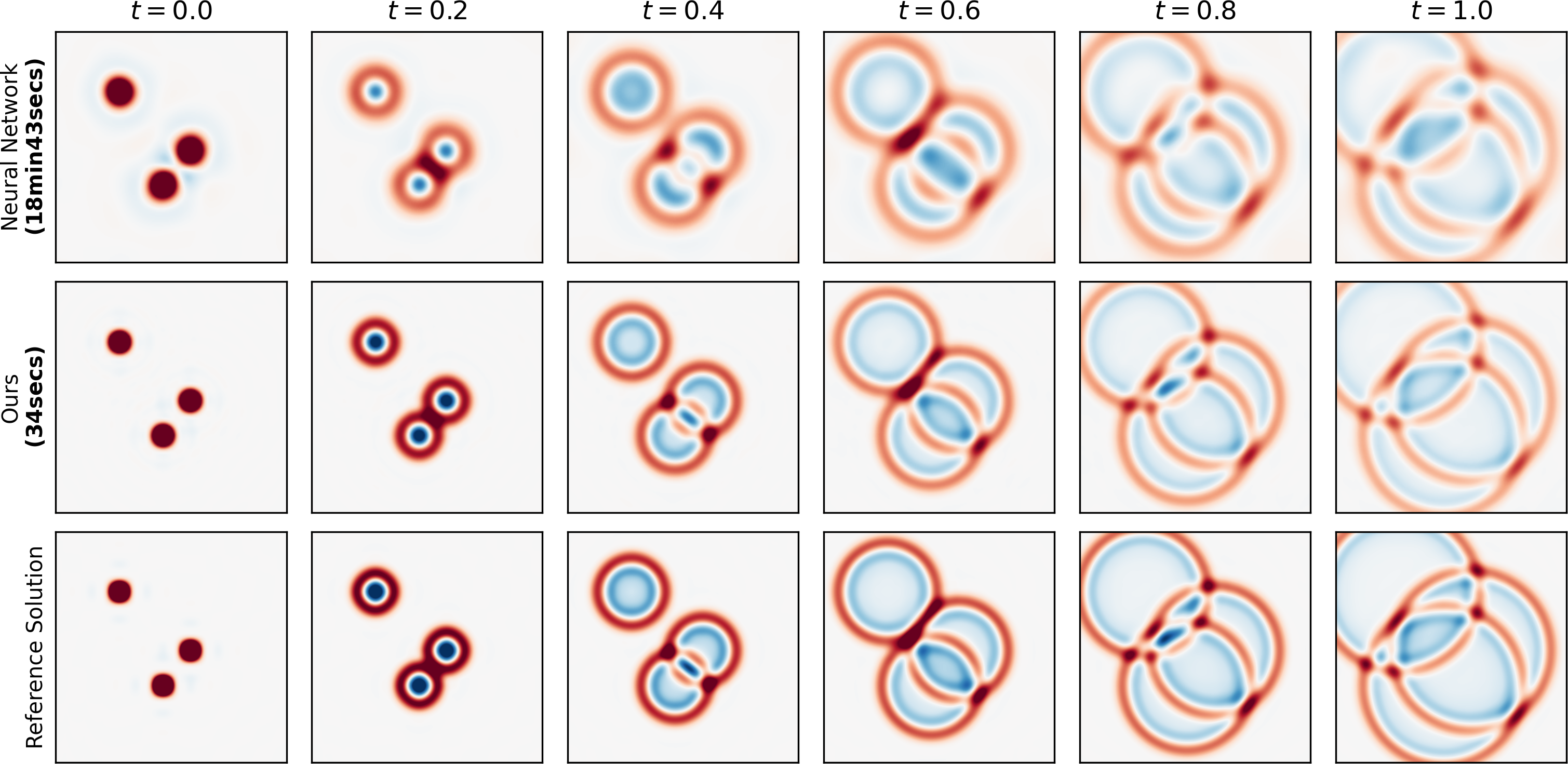}
    \vspace{-0.2cm}
\caption{\textbf{Our method solves PDEs.}
    Comparison of learned wave dynamics over time.
    Our adaptive FGD solution (middle) closely matches the reference solution (bottom), while achieving higher quality than the neural network baseline (top) with a much lower training time.}
    \label{fig:exp2}
    \vspace{-0.2cm}
\end{figure}

This section shows how our method can be applied to physics-informed machine learning tasks \citep{raissi2019physics, bizzi2025neusa}.
As an example, consider the wave equation for $f:\R\times\R^2\to\R$,
\begin{equation}\label{eq:wave-equation}
    \frac{\partial^2 f}{\partial t^2} = c^2 \Delta f,
    \qquad \text{with} \qquad
    \begin{cases}
        f(0, \textbf{x}) = h(\textbf{x}),
        \\
        \frac{\partial f}{\partial t}(0, \mathbf{x}) = 0;
    \end{cases}
\end{equation}
which models wave propagation phenomena such as sound, mechanical vibrations, and electromagnetic waves~\citep{evans2010partial,jackson1999classical,jensen2011computational,marques2025stable}.
Although solutions exist under broad conditions on the initial data $h$, explicit closed-form solutions are generally unavailable in realistic settings, requiring numerical approximation methods.

We can formulate solving \autoref{eq:wave-equation} as minimizing the following loss functional
\begin{align}
    L_\mathrm{PDE}(f) &:=\!\!
        \int \frac{1}{2} \left[ \frac{\partial^2 f}{\partial t^2} - c^2  \Delta f \right]^2 \!\!\dif \textbf{x} \dif t
 +\!\!\int \frac{1}{2} \left[ f(0, \textbf{x}) - h(\textbf{x}) \right]^2 \dif \textbf{x}
            +\!\! \int \frac{1}{2} \left[ \frac{\partial f}{\partial t}(0, \textbf{x}) \right]^2 \!\!\dif \textbf{x},
\end{align}
which vanishes if and only if $f$ solves \autoref{eq:wave-equation}.
The functional is defined over the Sobolev space $\H = H^2(\R^3)$ so that the second-order derivatives are well-defined.

Computing functional gradients in Sobolev spaces is generally challenging~\citep{fgd-random-gradientfree}.
However, the \emph{Fourier transform} of the gradient admits a closed-form expression:

\begin{proposition}\label{prop:gradient-pde}
    Let $\omega := \tau^2 - c^2(\xi^2+\zeta^2)$.
    The Fourier transform of the gradient of $L_\mathrm{PDE}$ is
    {\small
        \[\!\!\!
        [\widehat{\nabla L_\mathrm{PDE}(f)}](\tau,\xi,\zeta)
        \!=\!
        \frac{
            \omega^2 \widehat{f}(\tau,\xi,\zeta)
            + \frac{1}{2\pi}\int \widehat{f}(\tau',\xi,\zeta)\dif\tau'
            \!-\! \frac{1}{\sqrt{2\pi}}\widehat{h}(\xi,\zeta)
            + \frac{\tau}{2\pi}\left(\int \tau' \widehat{f}(\tau',\xi,\zeta)\dif\tau'\right)
        }{
            1 + (\tau^2+\xi^2+\zeta^2) + (\tau^2+\xi^2+\zeta^2)^2
        }.
    \]
    }
\end{proposition}

We therefore approximate gradients using adaptive uniform grids in frequency space.
Each grid defines a piecewise-constant approximation of the Fourier transform $\widehat{f}$, while the corresponding function $f$ is recovered through the inverse Fourier transform.
Moreover, the Fourier representation of the $H^2(\R^3)$ norm yields a straightforward upper bound on the gradient approximation error.

Figure~\ref{fig:exp2} shows the results.
Our adaptive method consistently outperforms FGD with fixed representations, which quickly plateaus due to approximation error.
It also surpasses the neural network baseline while reducing wall-clock training time by nearly two orders of magnitude.

\subsection{Learning radiance fields}

We next consider the inverse rendering task of reconstructing a 3D scene from posed images~\citep{mildenhall2021nerf}.
This can be framed as a functional optimization problem: given images with known camera parameters, the goal is to learn the scene as a density field and a view-dependent color field,
\[
    \sigma : \R^3 \to \R
    \qquad\text{and}\qquad
    c : \R^3 \times \mathbb{S}^2 \to \R^3.
\]
Given a camera ray $\gamma : [t_\near,t_\far]\to\R^3$ with viewing direction $\gamma_\omega$, where $t_\near$ and $t_\far$ denote integration bounds, we can render the scene via the volumetric rendering equation~\citep{max1995optical} with distributional antialiasing~\citep{cook1984distributed}, yielding the antialiased rendering operator $\widetilde{R}_{\sigma,c}(\gamma)$, which is differentiable w.r.t. $\sigma$ and $c$.
Then, given a set of pixels $(Y_{i,j})_{i=1,j=1}^{n,P}$ and corresponding rays $(\gamma_{i,j})_{i=1,j=1}^{n,P}$ originating from the respective cameras $i = 1, \ldots, n$, our goal is to optimize the loss
\begin{align}\label{eq:photometric}
  L_\mathrm{CV}(\sigma, c) := \frac{1}{n} \sum_{i=1}^n \frac{1}{P} \sum_{j=1}^P \frac{1}{2} \lVert \widetilde{R}_{\sigma,c}(\gamma_{i,j}) - Y_{i,j} \rVert^2,
\end{align}
which is well-defined for $\sigma \in L^2(\Omega)$ and $c \in L^2(\Omega \to \mathcal{H}^{\mathbb{S}})$, where $\Omega$ is a compact subset of $\R^3$ and $\mathcal{H}^\mathbb{S}$ is an RKHS with domain over the unit sphere $\mathbb{S}^2$.
This loss directly measures image reconstruction error and is the standard supervision signal for radiance-field optimization (cf. e.g. \citep{mildenhall2021nerf,kerbl20233d}).
It is highly nonconvex; nonetheless, it has functional gradients known in closed form, cf. Appendix~\ref{suppl:sec:functional-gradient-3d}, which we efficiently approximate by the means of uniform 3D grids with spherical harmonics for the color view-dependence.

We evaluate on the Ficus scene from the NeRF-Synthetic dataset~\citep{mildenhall2021nerf}, using the same camera poses and train/test split for all methods. All methods optimize the photometric loss in \eqref{eq:photometric}; we compare adaptive FGD with an Adam-trained neural network baseline and fixed-grid FGD variants. As shown in Figure~\ref{fig:radiance}, adaptive FGD achieves the lowest test loss and sharpest reconstructions by starting coarse and refining only when needed, whereas fixed-grid FGD plateaus at a resolution-dependent error floor and the neural baseline converges much more slowly.

\begin{figure}[t]
\centering
\includegraphics[width=\columnwidth]{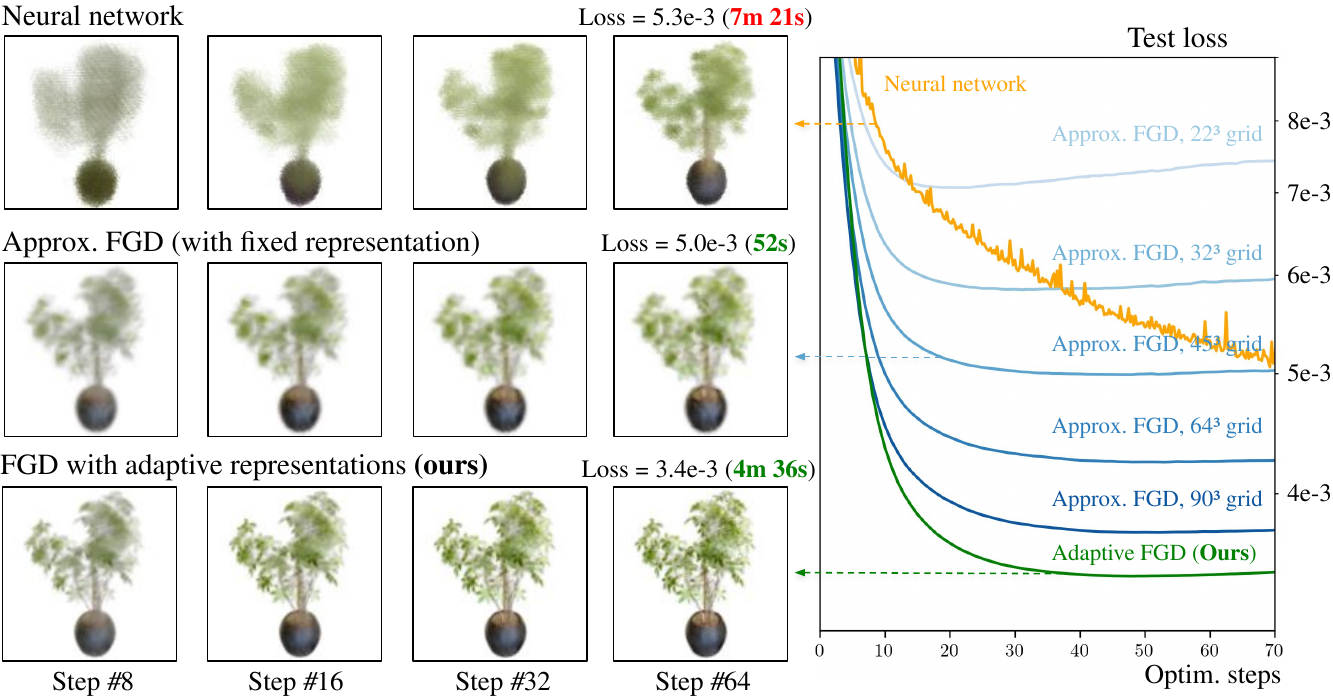}
\caption{\textbf{Our method solves inverse rendering.} We compare novel-view renderings, from viewpoints not used during training, for an Adam-trained neural network (top), fixed-grid FGD (middle), and adaptive FGD (bottom).
    Adaptive FGD progressively refines its representation, achieving a lower loss and sharper reconstructions.}
\label{fig:radiance}\end{figure}

\begin{ack}
    We'd like to thank Google for partially funding this work. We also acknowledge partial funding from CNPq and FAPERJ.
\end{ack}

\bibliographystyle{unsrtnat}
\bibliography{bibliography}

%%%%%%%%%%%%%%%%%%%%%%%%%%%%%%%%%%%%%%%%%%%%%%%%%%%%%%%%%%%%

\newpage
\appendix

\section{Proofs}

\begin{proof}[Proof of Lemma~\ref{lemma:sufficient-descent-lemma}]
    Let us derive the bound in the lemma.
    By $K$-smoothness, we have that
    \begin{align*}
        L(f_t - \eta g_t) - L(f_t)
        &\leq \D L(f_t; -\eta g_t) + \frac{K}{2} \lVert -\eta g_t \rVert_{\B}^2
        = -\eta \D L(f_t; g_t) + \frac{K \eta^2}{2} \lVert g_t \rVert_{\B}^2;
    \end{align*}
    Writing $g_t = \nabla L(f_t) - (\nabla L(f_t) - g_t) =: \nabla L(f_t) - e_t$, it follows:\footnote{Throughout, we follow the convention that $0/0 = 0$ and $a/0 = +\infty$ for all $a>0$.}
    \begin{align*}
        &\mathrel{\phantom{=}} -\eta \D L(f_t; g_t) + \frac{K \eta^2}{2} \lVert g_t \rVert_{\B}^2
        \\ &= -\eta \D L(f_t; \nabla L(f_t)) + \eta \D L(f_t; e_t) + \frac{K \eta^2}{2} \lVert \nabla L(f_t) - e_t \rVert_{\B}^2
        \\ &\leq -\eta \alpha \lVert \nabla L(f_t) \rVert_\B^2 + \eta \D L(f_t; e_t) + \frac{K \eta^2}{2} \lVert \nabla L(f_t) - e_t \rVert_{\B}^2
            \hspace{8.4em} \text{(Assumption~\ref{assumpt:h-descends-in-b})}
        \\ &\leq -\eta \alpha \lVert \nabla L(f_t) \rVert_\B^2 + \eta \lVert \D L(f_t; \cdot) \rVert_{\B^*} \lVert e_t \rVert_\B + \frac{K \eta^2}{2} \lVert \nabla L(f_t) - e_t \rVert_{\B}^2
            \hspace{4.5em} \text{(Hölder)}
        \\ &\leq -\eta \alpha \lVert \nabla L(f_t) \rVert_\B^2 + \eta \beta \lVert \nabla L(f_t) \rVert_\B \lVert e_t \rVert_\B + \frac{K \eta^2}{2} \lVert \nabla L(f_t) - e_t \rVert_{\B}^2
            \hspace{4.9em} \text{(Assumption~\ref{assumpt:grad-compat})}
        \\ &\leq -\eta \alpha \lVert \nabla L(f_t) \rVert_\B^2 + \eta \beta \lVert \nabla L(f_t) \rVert_\B \lVert e_t \rVert_\B + \frac{K \eta^2}{2} \Bigl( \lVert \nabla L(f_t) \rVert_\B + \lVert e_t \rVert_\B \Bigr)^2
            \hspace{1.7em}  \text{(Triangle ineq.)}
        \\ &= -\eta \alpha \lVert \nabla L(f_t) \rVert_\B^2 + ( \eta \beta + K \eta^2 ) \lVert \nabla L(f_t) \rVert_\B \lVert e_t \rVert_\B + \frac{K \eta^2}{2} \lVert \nabla L(f_t) \rVert_\B^2 + \frac{K \eta^2}{2} \lVert e_t \rVert_\B^2
        \\ &= -\eta \left( \alpha - \frac{K\eta}{2} - (\beta + K\eta) \frac{\lVert e_t \rVert_\B}{\lVert \nabla L(f_t) \rVert_\B} - \frac{K\eta}{2} \frac{\lVert e_t \rVert_\B^2}{\lVert \nabla L(f_t) \rVert_\B^2} \right) \lVert \nabla L(f_t) \rVert_\B^2;
    \end{align*}
    to simplify the expression, let us assume that ${\lVert \nabla L(f_t) - g_t \rVert_\B}/{\lVert \nabla L(f_t) \rVert_\B} = {\lVert e_t \rVert_\B}/{\lVert \nabla L(f_t) \rVert_\B} \leq 1$. This allows us to use the fact that $x^2 \leq x$ for $\lvert x \rvert \leq 1$ to write
    \looseness=-1
    \begin{align*}
        &\mathrel{\phantom{=}} -\eta \left( \alpha - \frac{K\eta}{2} - (\beta + K\eta) \frac{\lVert \nabla L(f_t) - g_t \rVert_\B}{\lVert \nabla L(f_t) \rVert_\B} - \frac{K\eta}{2} \frac{\lVert \nabla L(f_t) - g_t \rVert_\B^2}{\lVert \nabla L(f_t) \rVert_\B^2} \right) \lVert \nabla L(f_t) \rVert_\B^2;
        \\ &\leq -\eta \left( \alpha - \frac{K\eta}{2} - (\beta + K\eta) \frac{\lVert \nabla L(f_t) - g_t \rVert_\B}{\lVert \nabla L(f_t) \rVert_\B} - \frac{K\eta}{2} \frac{\lVert \nabla L(f_t) - g_t \rVert_\B}{\lVert \nabla L(f_t) \rVert_\B} \right) \lVert \nabla L(f_t) \rVert_\B^2;
        \\ &= -\eta \left( \alpha - \frac{K\eta}{2} - \left(\beta + \frac{3}{2} K\eta\right) \frac{\lVert \nabla L(f_t) - g_t \rVert_\B}{\lVert \nabla L(f_t) \rVert_\B} \right) \lVert \nabla L(f_t) \rVert_\B^2.
    \end{align*}
    As long as the expression within the parentheses is positive, we ensure descent.
    However, the relative error ${\lVert \nabla L(f_t) - g_t \rVert_\B}/{\lVert \nabla L(f_t) \rVert_\B}$ is a bit tricky to manage in practice since the denominator involves the $\lVert \nabla L(f_t) \rVert_\B$, which can be hard to tightly lower bound.
    Fortunately, we can replace it with $\lVert g_t \rVert_\B$, which can be generally be computed exactly; indeed, by the reverse triangle inequality,\footnote{$\bigl\lvert \lVert u \rVert - \lVert v \rVert \bigr\rvert \leq \lVert u - v \rVert$; as an immediate consequence, $\lVert u \rVert - \lVert v \rVert \leq \lVert u - v \rVert$.}
    \begin{align*}
        \lVert g_t \rVert_\B - \lVert \nabla L(f_t) \rVert_\B
        \leq \lVert g_t - \nabla L(f_t) \rVert_\B
        = \frac{\lVert g_t - \nabla L(f_t) \rVert_\B}{\lVert g_t \rVert_\B} \lVert g_t \rVert_\B;
    \end{align*}
    rearranging, we obtain that
    \[ \lVert \nabla L(f_t) \rVert_\B \geq \left( 1 - \frac{\lVert g_t - \nabla L(f_t) \rVert_\B}{\lVert g_t \rVert_\B} \right) \lVert g_t \rVert_\B, \]
    and thus
    \[ \frac{\lVert g_t - \nabla L(f_t) \rVert_\B}{\lVert \nabla L(f_t) \rVert_\B} \leq \frac{\lVert g_t - \nabla L(f_t) \rVert_\B}{\left( 1 - \frac{\lVert g_t - \nabla L(f_t) \rVert_\B}{\lVert g_t \rVert_\B} \right) \lVert g_t \rVert_\B} = \frac{\RelErr(g_t, \nabla L(f_t))}{1 - \RelErr(g_t, \nabla L(f_t))}, \]
    where
    \[ \RelErr(g_t, \nabla L(f_t)) := \frac{\lVert g_t - \nabla L(f_t) \rVert_\B}{\lVert g_t \rVert_\B}; \]
    finally, to guarantee that ${\lVert g_t - \nabla L(f_t) \rVert_\B}/{\lVert \nabla L(f_t) \rVert_\B} \leq 1$ as required, it suffices to require $\RelErr(g_t, \nabla L(f_t)) \leq 1/2$, as $\frac{1}{2} / (1 - \frac{1}{2}) = 1$.
    Putting it all together, we obtain the desired bound.
\end{proof}

\begin{proof}[Proof of Proposition~\ref{prop:convergence-to-stationary-point}]
    It follows:
    \begin{align*}
        L(f_T) - L(f_0)
        &= \sum_{t=0}^{T-1} \left( L(f_{t+1}) - L(f_t) \right)
        \\ &\leq -\eta \sum_{t=0}^{T-1} \left[ \alpha - \frac{K \eta}{2} - \left( \beta + \frac{3}{2} K \eta \right) \frac{\RelErr(g_t, \nabla L(f_t))}{1-\RelErr(g_t, \nabla L(f_t))} \right] \lVert \nabla L(f_t) \rVert_\B^2
        \\ &\leq -\eta \left[ \alpha - \frac{K \eta}{2} - \left( \beta + \frac{3}{2} K \eta \right) \frac{\epsilon}{1-\epsilon} \right] \sum_{t=0}^{T-1} \lVert \nabla L(f_t) \rVert_\B^2,
    \end{align*}
    where the first inequality holds by Lemma~\ref{lemma:sufficient-descent-lemma}, and the second by the assumption that $\RelErr(g_t, \nabla L(f_t)) \leq \epsilon$ for all $t = 1, \ldots, T$.

    Rearranging and scaling by $1/T$ on both sides, we get
    \begin{align*}
        \eta \left[ \alpha - \frac{K \eta}{2} - \left( \beta + \frac{3}{2} K \eta \right) \frac{\epsilon}{1-\epsilon} \right] \frac{1}{T} \sum_{t=0}^{T-1} \lVert \nabla L(f_t) \rVert_\B^2
        \leq \frac{L(f_0) - L(f_T)}{T}
        \leq \frac{L(f_0) - L^\star}{T};
    \end{align*}
    dividing by $r = \left[ \alpha - \frac{K \eta}{2} - \left( \beta + \frac{3}{2} K \eta \right) \frac{\epsilon}{1-\epsilon} \right]$ on both sides yields the desired bound; to show that $r>0$, simply note that
    \begin{align*}
        &\mathrel{\phantom{=}} \alpha - \frac{K \eta}{2} - \left( \beta + \frac{3}{2} K \eta \right) \frac{\epsilon}{1-\epsilon}
        \\ &= \alpha - \frac{K}{2} \eta - \beta \frac{\epsilon}{1-\epsilon} - \frac{3}{2} \eta K \frac{\epsilon}{1-\epsilon}
        \\ &= -\left( \frac{3}{2} K \frac{\epsilon}{1-\epsilon} + \frac{K}{2} \right) \eta + \left( \alpha - \beta \frac{\epsilon}{1-\epsilon} \right) > 0
        \\ &\iff \left( \alpha - \beta \frac{\epsilon}{1-\epsilon} \right) > \left( \frac{3}{2} K \frac{\epsilon}{1-\epsilon} + \frac{K}{2} \right) \eta
        \\ &\iff \eta < \frac{\alpha - \beta \frac{\epsilon}{1-\epsilon}}{\frac{3}{2} K \frac{\epsilon}{1-\epsilon} + \frac{K}{2}}
                      = \frac{\alpha (1-\epsilon) - \beta \epsilon}{\frac{3}{2} K \epsilon + \frac{K}{2} (1-\epsilon)}
                      = \frac{\alpha - (\alpha + \beta) \epsilon}{ \frac{K}{2} \left( 2\epsilon + 1 \right) }
                      = \frac{2 (\alpha - (\alpha + \beta) \epsilon)}{ K \left( 2\epsilon + 1 \right) },
    \end{align*}
    which is positive thanks to the upper bound on $\epsilon$.
\end{proof}

\begin{proof}[Proof of Proposition~\ref{prop:convergence-to-global-minimizer}]
    We have that
    \begin{align*}
        L(f_{t+1}) - L(f_t)
        &\leq -\eta \left[ \alpha - \frac{K \eta}{2} - \left( \beta + \frac{3}{2} K \eta \right) \frac{\RelErr(g_t, \nabla L(f_t))}{1-\RelErr(g_t, \nabla L(f_t))} \right] \lVert \nabla L(f_t) \rVert_\B^2
        \\ &\leq -\eta \left[ \alpha - \frac{K \eta}{2} - \left( \beta + \frac{3}{2} K \eta \right) \frac{\epsilon}{1-\epsilon} \right] \lVert \nabla L(f_t) \rVert_\B^2
        \\ &= -\eta r \, \lVert \nabla L(f_t) \rVert_\B^2
        \\ &\leq -\eta \beta^{-2} r \, \lVert \D L(f_t; \cdot) \rVert_{\B^*}^2
        \\ &\leq -2 \eta \beta^{-2} \mu r \, \bigl( L(f_t) - L^\star \bigr),
    \end{align*}
    where the first inequality holds by Lemma~\ref{lemma:sufficient-descent-lemma}, the second by the assumption that $\RelErr(g_t, \nabla L(f_t)) \leq \epsilon$ for all $t = 1, \ldots, T$, and the third by the Polyak-Łojasiewicz assumption.
    Rearranging:
    \begin{align*}
        &\mathrel{\phantom{\iff}} L(f_{t+1}) - L(f_t) \leq -2 \eta \beta^{-2} \mu r \, \bigl( L(f_t) - L^\star \bigr)
        \\ &\iff \bigl( L(f_{t+1}) - L(f_t) \bigr) + \bigl( L(f_t) - L^\star \bigr) \leq -2 \eta \beta^{-2} \mu r \, \bigl( L(f_t) - L^\star \bigr) + \bigl( L(f_t) - L^\star \bigr)
        \\ &\iff L(f_{t+1}) - L^\star \leq (1 - 2 \eta \beta^{-2} \mu r) \, \bigl( L(f_t) - L^\star \bigr);
    \end{align*}
    chaining these over $t = 0, \ldots, T-1$ yields
    \[
        L(f_T) - L^\star
        \leq [1 - 2 \eta \beta^{-2} \mu r]^T \, \bigl( L(f_0) - L^\star \bigr). \qedhere
    \]
\end{proof}

\begin{proof}[Proof of Lemma~\ref{lemma:approx-error-limit}]
    First, note that by the sandwich theorem,
    \[ \lim_{n\to\infty} U(g_n, \nabla L(f)) = 0 \implies g_n \to \nabla L(f); \]
    Then, by continuity,
    \begin{align*}
        \lim_{n \to \infty} \RelErr(g_n, \nabla L(f))
        &= \lim_{n \to \infty} \frac{\lVert \nabla L(f) - g_n \rVert_\B}{\lVert g_n \rVert_\B}
        = \frac{\lim_{n \to \infty} \lVert \nabla L(f) - g_n \rVert_\B}{\lim_{n \to \infty} \lVert g_n \rVert_\B}
        \\ &= \frac{\lVert \nabla L(f) - \lim_{n \to \infty} g_n \rVert_\B}{\lVert \lim_{n \to \infty} g_n \rVert_\B}
        = \frac{\lVert \nabla L(f) - \nabla L(f) \rVert_\B}{\lVert \nabla L(f) \rVert_\B}
        = 0,
    \end{align*}
    and so
    \begin{align*}
        \lim_{n \to \infty} \frac{\RelErr(g_n, \nabla L(f))}{1 - \RelErr(g_n, \nabla L(f))}
        = \frac{\lim_{n \to \infty} \RelErr(g_n, \nabla L(f))}{1 - \lim_{n \to \infty} \RelErr(g_n, \nabla L(f))}
        = \frac{0}{1 - 0} = 0.
    \end{align*}
    The result then follows by the definition of the limit of a sequence: for any $\epsilon > 0$, there exists some $n$ such that the bound is satisfied for all $n' \geq n$ -- which, in particular, must hold for $n$ itself.
\end{proof}

\begin{proof}[Proof of Theorem~\ref{thm:main-theorem}]
    Let us start by establishing claim (i).
    For Algorithm~\ref{alg:main-algorithm} to perform a gradient descent step, we must have that
    \[ (1 + \epsilon) U_t < \epsilon \lVert g_t \rVert_\B, \]
    i.e., that
    \[ \frac{U_t}{\lVert g_t \rVert_\B} \leq \frac{\epsilon}{1+\epsilon}; \]
    and since $\lVert g_t - \nabla L(f_t) \rVert_\B \leq U_t$, we have that
    \[ \RelErr(g_t, \nabla L(f_t)) = \frac{\lVert g_t - \nabla L(f_t) \rVert_\B}{\lVert g_t \rVert_\B} \leq \frac{\epsilon}{1+\epsilon} < \frac{1}{2}, \]
    where the last inequality follows by the fact that $\epsilon \in (0, 1)$.
    Additionally, by Lemma~\ref{lemma:approx-error-limit}, the while loop is guaranteed to terminate, allowing the algorithm to proceed.

    Claims (ii) and (iii) then follow immediately by applying Propositions~\ref{prop:convergence-to-stationary-point} and \ref{prop:convergence-to-global-minimizer}.
\end{proof}

\subsection{Computation of Functional Derivatives for the Experiments}

\subsubsection{Function fitting (Figure~\ref{fig:teaser})}

We consider here the loss $L_\mathrm{FIT} : L^2([0,1]^2) \to \R$ given by
\[ L_\mathrm{FIT}(f) := \frac{1}{2} \lVert f - f^\star \rVert_{L^2}^2 = \iint \frac{1}{2} \bigl[ f(x, y) - f^\star(x, y) \bigr]^2 \dif x \dif y. \]

It holds:

\begin{proposition}
    The functional $L_\mathrm{FIT} : L^2([0,1]^2) \to \R$ has gradients given by
    \[ \nabla L_\mathrm{FIT}(f)(x, y) = f(x, y) - f^\star(x, y), \]
    and is 1-Polyak-Łojasiewicz.
\end{proposition}
\begin{proof}
    The directional derivative is
    \begin{align*}
        \D L_\mathrm{FIT}(f; \vec{f})
        &= \left[ \frac{\dif}{\dif \delta} L_\mathrm{FIT}(f + \delta \vec{f}) \right]_{\delta=0}
        \\ &= \left[ \frac{\dif}{\dif \delta} \iint \frac{1}{2} \bigl[ f(x, y) + \delta \vec{f}(x, y) - f^\star(x, y) \bigr]^2 \dif x \dif y \right]_{\delta=0}
        \\ &= \left[ \iint \frac{\dif}{\dif \delta} \frac{1}{2} \bigl[ f(x, y) + \delta \vec{f}(x, y) - f^\star(x, y) \bigr]^2 \dif x \dif y \right]_{\delta=0}
        \\ &= \iint \bigl[ f(x, y) - f^\star(x, y) \bigr] \vec{f}(x, y) \dif x \dif y
        \\ &= \left\langle f - f^\star, \vec{f} \right\rangle_{L^2},
    \end{align*}
    and so
    \[ \nabla L_\mathrm{FIT}(f)(x, y) = f(x, y) - f^\star(x, y). \]
    To show that the loss is 1-Polyak-Łojasiewicz, note:
    \[
        L_\mathrm{FIT}(f) - \inf_{f' \in L^2([0,1]^2)} L_\mathrm{FIT}(f')
        = L_\mathrm{FIT}(f) - 0
        = \frac{1}{2} \lVert f - f^\star \rVert_{L^2}^2
        = \frac{1}{2} \lVert \nabla L_\mathrm{FIT}(f) \rVert_{L^2}^2.
        \qedhere
    \]
\end{proof}

\subsubsection{Regression in an RKHS (Section~\ref{sec:exp1})}

\begin{proof}[Proof of Proposition~\ref{prop:gradient-kernel}]
    The directional derivative is
    \begin{align*}
        \D L_\mathrm{REG}(f; \vec{f})
        &= \left[ \frac{\dif}{\dif \delta} L_\mathrm{REG}(f + \delta \vec{f}) \right]_{\delta=0}
        \\ &= \left[ \frac{\dif}{\dif \delta} \frac{1}{n} \sum_{i=1}^n \ell(f(X_i) + \delta \vec{f}(X_i), Y_i) \right]_{\delta=0}
        \\ &= \left[ \frac{\dif}{\dif \delta} \frac{1}{n} \sum_{i=1}^n \ell(\langle f + \delta \vec{f}, K_{X_i} \rangle_{\mathcal{H}_K}, Y_i) \right]_{\delta=0}
        \\ &= \left[ \frac{1}{n} \sum_{i=1}^n \frac{\dif}{\dif \delta} \ell(\langle f + \delta \vec{f}, K_{X_i} \rangle_{\mathcal{H}_K}, Y_i) \right]_{\delta=0}
        \\ &= \left[ \frac{1}{n} \sum_{i=1}^n \ell'(\langle f + \delta \vec{f}, K_{X_i} \rangle_{\mathcal{H}_K}, Y_i) \langle \vec{f}, K_{X_i} \rangle_{\mathcal{H}_K} \right]_{\delta=0}
        \\ &= \frac{1}{n} \sum_{i=1}^n \ell'(\langle f, K_{X_i} \rangle_{\mathcal{H}_K}, Y_i) \langle \vec{f}, K_{X_i} \rangle_{\mathcal{H}_K}
        \\ &= \frac{1}{n} \sum_{i=1}^n \ell'(f(X_i), Y_i) \langle \vec{f}, K_{X_i} \rangle_{\mathcal{H}_K}
        \\ &= \left\langle \vec{f}, \frac{1}{n} \sum_{i=1}^n \ell'(f(X_i), Y_i) K_{X_i} \right\rangle_{\mathcal{H}_K},
    \end{align*}
    and thus
    \[ \nabla L_\mathrm{REG}(f)(x) = \frac{1}{n} \sum_{i=1}^n \ell'(f(X_i), Y_i) K(X_i, x). \]
    To establish that the loss is Polyak-Łojasiewicz, note that the loss depends on $f$ only through its evaluations at the data points. Let $v \in \R^n$ with $v_i = f(X_i)$, and let $g(v) = \frac{1}{n} \sum_{i=1}^n \ell(v_i, Y_i)$. Because $\ell$ is $\mu$-strongly convex in its first argument, $g$ is strongly convex with respect to the Euclidean norm with parameter $\mu/n$.

    Letting $c_i = \ell'(f(X_i), Y_i)$, the gradient is $\nabla g(v) = \frac{1}{n} c$. By standard strong convexity, the suboptimality is bounded by the gradient norm:
    \[ L_\mathrm{REG}(f) - \inf_{f' \in \mathcal{H}_K} L_\mathrm{REG}(f') \leq \frac{1}{2(\mu/n)} \lVert \nabla g(v) \rVert_2^2 = \frac{n}{2\mu} \sum_{i=1}^n \left( \frac{1}{n} c_i \right)^2 = \frac{1}{2\mu n} \lVert c \rVert_2^2. \]
    We must now bound $\lVert c \rVert_2^2$ using the $L^\infty$ norm of the extended RKHS gradient. Recall that the continuous extension of the gradient to $\mathcal{B} = L^\infty$ is the function:
    \[ \nabla L_\mathrm{REG}(f)(\cdot) = \frac{1}{n} \sum_{j=1}^n c_j K(X_j, \cdot). \]
    Let $K \in \R^{n \times n}$ be the empirical kernel matrix, and let $g_X \in \R^n$ be the evaluation of this gradient at the training points, so $g_X^{(i)} = \nabla L_\mathrm{REG}(f)(X_i)$. In matrix form, we have $g_X = \frac{1}{n} K c$.

    Assuming the kernel matrix $K$ is strictly positive definite, we can invert this relationship to get $c = n K^{-1} g_X$. Bounding the Euclidean norm of $c$ via the smallest eigenvalue $\lambda_{\min}(K) > 0$ yields:
    \[ \lVert c \rVert_2 \leq n \lVert K^{-1} \rVert_2 \lVert g_X \rVert_2 = \frac{n}{\lambda_{\min}(K)} \lVert g_X \rVert_2. \]
    Next, we bound the discrete Euclidean norm $\lVert g_X \rVert_2$ by the supremum norm over the entire space. Since $\lvert g_X^{(i)} \rvert \leq \sup_x \lvert \nabla L_\mathrm{REG}(f)(x) \rvert = \lVert \nabla L_\mathrm{REG}(f) \rVert_{L^\infty}$ for all $i$, we have:
    \[ \lVert g_X \rVert_2^2 = \sum_{i=1}^n \bigl( g_X^{(i)} \bigr)^2 \leq n \lVert \nabla L_\mathrm{REG}(f) \rVert_{L^\infty}^2 \implies \lVert g_X \rVert_2 \leq \sqrt{n} \lVert \nabla L_\mathrm{REG}(f) \rVert_{L^\infty}. \]
    Substituting this back into the bound for $\lVert c \rVert_2$, we get:
    \[ \lVert c \rVert_2 \leq \frac{n^{3/2}}{\lambda_{\min}(K)} \lVert \nabla L_\mathrm{REG}(f) \rVert_{L^\infty} \implies \lVert c \rVert_2^2 \leq \frac{n^3}{\lambda_{\min}^2(K)} \lVert \nabla L_\mathrm{REG}(f) \rVert_{L^\infty}^2. \]
    Finally, plugging this into our suboptimality bound yields:
    \[ L_\mathrm{REG}(f) - L^\star \leq \frac{1}{2\mu n} \left( \frac{n^3}{\lambda_{\min}^2(K)} \lVert \nabla L_\mathrm{REG}(f) \rVert_{L^\infty}^2 \right) = \frac{n^2}{2\mu \lambda_{\min}^2(K)} \lVert \nabla L_\mathrm{REG}(f) \rVert_{L^\infty}^2. \]

    Finally, to show that $L_\mathrm{REG}$ is $K$-smooth:
    let $f, f' \in L^\infty$. Using the definition of $L_\mathrm{REG}$ and applying the $K$-smoothness of $\ell$ to each data point $X_i$, we have:
    \begin{align*}
        L_\mathrm{REG}(f) 
        &= \frac{1}{n} \sum_{i=1}^n \ell(f(X_i), Y_i) \\
        &\leq \frac{1}{n} \sum_{i=1}^n \left[ \ell(f'(X_i), Y_i) + \ell'(f'(X_i), Y_i) \bigl( f(X_i) - f'(X_i) \bigr) + \frac{K}{2} \bigl( f(X_i) - f'(X_i) \bigr)^2 \right] \\
        &= \left( \frac{1}{n} \sum_{i=1}^n \ell(f'(X_i), Y_i) \right) + \left( \frac{1}{n} \sum_{i=1}^n \ell'(f'(X_i), Y_i) \bigl( f(X_i) - f'(X_i) \bigr) \right) \\
        &\qquad + \frac{K}{2n} \sum_{i=1}^n \bigl( f(X_i) - f'(X_i) \bigr)^2 \\
        &= L_\mathrm{REG}(f') + \D L_\mathrm{REG}(f'; f - f') + \frac{K}{2n} \sum_{i=1}^n \bigl( f(X_i) - f'(X_i) \bigr)^2.
    \end{align*}
    Next, we bound the quadratic term using the definition of the $L^\infty$ norm. Since $\lvert f(X_i) - f'(X_i) \rvert \leq \sup_{x} \lvert f(x) - f'(x) \rvert = \lVert f - f' \rVert_{L^\infty}$ for all $i$, it follows that:
    \[
        \frac{K}{2n} \sum_{i=1}^n \bigl( f(X_i) - f'(X_i) \bigr)^2 \leq \frac{K}{2n} \sum_{i=1}^n \lVert f - f' \rVert_{L^\infty}^2 = \frac{K}{2} \lVert f - f' \rVert_{L^\infty}^2.
    \]
    Substituting this back into our inequality yields:
    \[
        L_\mathrm{REG}(f) \leq L_\mathrm{REG}(f') + \D L_\mathrm{REG}(f'; f - f') + \frac{K}{2} \lVert f - f' \rVert_{L^\infty}^2. \qedhere
    \]
\end{proof}

\begin{proposition}\label{prop:gradient-bridging-rkhs}
    Let $L_\mathrm{REG}$ be the regression functional over $\mathcal{H}_K$, and assume the empirical kernel matrix $K \in \R^{n \times n}$ over the training data is strictly positive definite with minimum eigenvalue $\lambda_{\min}(K) > 0$. Furthermore, assume the kernel is bounded, such that $\sup_{x, x'} \lvert K(x, x') \rvert \leq \kappa$. Then, evaluating the gradients over the Banach space $\mathcal{B} = L^\infty$, the loss $L_\mathrm{REG}$ satisfies:
    \begin{enumerate}
        \item[(i)] \textbf{Gradient compatibility (Assumption~\ref{assumpt:grad-compat}):} with constant $\beta = n \lambda_{\min}^{-1}(K)$.
        \item[(ii)] \textbf{$\mathcal{H}$ descends in $\mathcal{B}$ (Assumption~\ref{assumpt:h-descends-in-b}):} with constant $\alpha = \frac{\lambda_{\min}(K)}{n \kappa^2}$.
    \end{enumerate}
\end{proposition}

\begin{proof}
    As established previously, let $c_i = \ell'(f(X_i), Y_i)$ for $i=1, \dots, n$, such that the continuous extension of the gradient to $\mathcal{B} = L^\infty$ is $\nabla L_\mathrm{REG}(f)(x) = \frac{1}{n} \sum_{i=1}^n c_i K(X_i, x)$. The directional derivative for any $\vec{f} \in L^\infty$ is $\D L_\mathrm{REG}(f; \vec{f}) = \frac{1}{n} \sum_{i=1}^n c_i \vec{f}(X_i)$.

    For gradient compatibility,
    we must show that $\lVert \D L_\mathrm{REG}(f; \cdot) \rVert_{(L^\infty)^*} \leq \beta \lVert \nabla L_\mathrm{REG}(f) \rVert_{L^\infty}$.
    Recall that the dual norm of the directional derivative over $L^\infty$ is exactly the scaled $L_1$ norm of the coefficients:
    \[
        \lVert \D L_\mathrm{REG}(f; \cdot) \rVert_{(L^\infty)^*} = \frac{1}{n} \lVert c \rVert_1.
    \]
    Let $g_X \in \R^n$ be the vector of gradient evaluations at the training points, so $g_X^{(i)} = \nabla L_\mathrm{REG}(f)(X_i)$. In matrix form, $g_X = \frac{1}{n} K c$, which gives $c = n K^{-1} g_X$. 
    By the definition of the supremum norm, the maximum absolute evaluation of the gradient over the entire domain bounds the maximum absolute evaluation at the training points: $\lVert g_X \rVert_\infty \leq \lVert \nabla L_\mathrm{REG}(f) \rVert_{L^\infty}$.

    It then follows:
    \begin{align*}
        \lVert c \rVert_1 
        \leq \sqrt{n} \lVert c \rVert_2 
        &= \sqrt{n} \lVert n K^{-1} g_X \rVert_2 \\
        &\leq n^{3/2} \lVert K^{-1} \rVert_2 \lVert g_X \rVert_2 \\
        &\leq n^{3/2} \lambda_{\min}^{-1}(K) \left( \sqrt{n} \lVert g_X \rVert_\infty \right) \\
        &= n^2 \lambda_{\min}^{-1}(K) \lVert g_X \rVert_\infty.
    \end{align*}
    Substituting this into the dual norm equation and applying $\lVert g_X \rVert_\infty \leq \lVert \nabla L_\mathrm{REG}(f) \rVert_{L^\infty}$ yields:
    \[
        \lVert \D L_\mathrm{REG}(f; \cdot) \rVert_{(L^\infty)^*} \leq \frac{1}{n} \left( n^2 \lambda_{\min}^{-1}(K) \lVert \nabla L_\mathrm{REG}(f) \rVert_{L^\infty} \right) = n \lambda_{\min}^{-1}(K) \lVert \nabla L_\mathrm{REG}(f) \rVert_{L^\infty}.
    \]
    Thus, gradient compatibility holds with $\beta = n \lambda_{\min}^{-1}(K)$.

    To establish that $\mathcal{H}$ descends in $\mathcal{B}$,
    we must show $\D L_\mathrm{REG}(f; \nabla L_\mathrm{REG}(f)) \geq \alpha \lVert \nabla L_\mathrm{REG}(f) \rVert_{L^\infty}^2$.
    Evaluating the directional derivative in the direction of the gradient yields:
    \[
        \D L_\mathrm{REG}(f; \nabla L_\mathrm{REG}(f)) = \frac{1}{n} \sum_{i=1}^n c_i \nabla L_\mathrm{REG}(f)(X_i) = \frac{1}{n} c^\top \left( \frac{1}{n} K c \right) = \frac{1}{n^2} c^\top K c.
    \]
    We now require an upper bound for $\lVert \nabla L_\mathrm{REG}(f) \rVert_{L^\infty}^2$ in terms of $c^\top K c$. Using the triangle inequality and the kernel bound $\kappa$:
    \[
        \lVert \nabla L_\mathrm{REG}(f) \rVert_{L^\infty} = \sup_{x} \left\lvert \frac{1}{n} \sum_{i=1}^n c_i K(X_i, x) \right\rvert \leq \frac{1}{n} \sum_{i=1}^n \lvert c_i \rvert \sup_{x} \lvert K(X_i, x) \rvert \leq \frac{\kappa}{n} \lVert c \rVert_1.
    \]
    Squaring both sides and applying the norm inequality $\lVert c \rVert_1^2 \leq n \lVert c \rVert_2^2$, we have:
    \[
        \lVert \nabla L_\mathrm{REG}(f) \rVert_{L^\infty}^2 \leq \frac{\kappa^2}{n^2} \lVert c \rVert_1^2 \leq \frac{\kappa^2}{n} \lVert c \rVert_2^2.
    \]
    By the Rayleigh quotient for positive definite matrices, $c^\top K c \geq \lambda_{\min}(K) \lVert c \rVert_2^2$, meaning $\lVert c \rVert_2^2 \leq \lambda_{\min}^{-1}(K) c^\top K c$. Substituting this gives:
    \[
        \lVert \nabla L_\mathrm{REG}(f) \rVert_{L^\infty}^2 \leq \frac{\kappa^2}{n \lambda_{\min}(K)} c^\top K c.
    \]
    Rearranging to isolate the quadratic form $c^\top K c$:
    \[
        \frac{1}{n^2} c^\top K c \geq \frac{\lambda_{\min}(K)}{n \kappa^2} \lVert \nabla L_\mathrm{REG}(f) \rVert_{L^\infty}^2.
    \]
    Since the left side is exactly $\D L_\mathrm{REG}(f; \nabla L_\mathrm{REG}(f))$, the condition is satisfied with $\alpha = \frac{\lambda_{\min}(K)}{n \kappa^2}$.
\end{proof}

\subsubsection{Solving Partial Differential Equations (Section~\ref{sec:exp2})}

\begin{proof}[Proof of Proposition~\ref{prop:gradient-pde}]
    The directional derivative is
    {\small
    \begin{align*}
        \D L_\mathrm{PDE}(f; \vec{f})
        &= \left[ \frac{\dif}{\dif \delta} L_\mathrm{PDE}(f + \delta \vec{f}) \right]_{\delta=0}
        \\ &= \Biggl[ \frac{\dif}{\dif \delta} \iiint \frac{1}{2} \left[ \frac{\partial^2 f}{\partial t^2} + \delta \frac{\partial^2 \vec{f}}{\partial t^2} - c^2 \left( \frac{\partial^2 f}{\partial x^2} + \frac{\partial^2 f}{\partial y^2} \right) - \delta c^2 \left( \frac{\partial^2 \vec{f}}{\partial x^2} + \frac{\partial^2 \vec{f}}{\partial y^2} \right) \right]^2 \dif x \dif y \dif t
            \\ &\qquad + \frac{\dif}{\dif \delta} \iint \frac{1}{2} \left[ f(0, x, y) + \delta \vec{f}(0, x, y) - h(x, y) \right]^2 \dif x \dif y
            \\ &\qquad + \frac{\dif}{\dif \delta} \iint \frac{1}{2} \left[ \frac{\partial f}{\partial t}(0, x, y) + \delta \frac{\partial \vec{f}}{\partial t}(0, x, y) \right]^2 \dif x \dif y \Biggr]_{\delta=0}
        \\ &= \iiint \left[ \frac{\partial^2 f}{\partial t^2} - c^2 \left( \frac{\partial^2 f}{\partial x^2} + \frac{\partial^2 f}{\partial y^2} \right) \right] \left[ \frac{\partial^2 \vec{f}}{\partial t^2} - c^2 \left( \frac{\partial^2 \vec{f}}{\partial x^2} + \frac{\partial^2 \vec{f}}{\partial y^2} \right) \right] \dif x \dif y \dif t
            \\ &\qquad + \iint \left[ f(0, x, y) - h(x, y) \right] \vec{f}(0, x, y) \dif x \dif y
            \\ &\qquad + \iint \frac{\partial f}{\partial t}(0, x, y) \, \frac{\partial \vec{f}}{\partial t}(0, x, y) \dif x \dif y.
    \end{align*}
    }
    Now, we have to rewrite this as an inner product on $H^2(\R^3)$.
    Recall how the Sobolev inner product can be given in terms of the Fourier transform \citep{evans}:
    \begin{align*}
        \langle f, g \rangle_{H^2(\R^3)}
        &= \int \left( 1 + \lVert \xi \rVert^2 + \lVert \xi \rVert^4 \right) \widehat{f}(\xi) \, \overline{\widehat{g}(\xi)} \dif \xi.
    \end{align*}
    Therefore, let us rewrite the directional derivative to be of this form. To start, we will make it so that all references to $f$ are through its Fourier transform, each term at a time:
    {\small
    \begin{align*}
        & \iiint \left[ \frac{\partial^2 f}{\partial t^2} - c^2 \left( \frac{\partial^2 f}{\partial x^2} + \frac{\partial^2 f}{\partial y^2} \right) \right] \left[ \frac{\partial^2 \vec{f}}{\partial t^2} - c^2 \left( \frac{\partial^2 \vec{f}}{\partial x^2} + \frac{\partial^2 \vec{f}}{\partial y^2} \right) \right] \dif x \dif y \dif t
        \\ &= \iiint \left[ i^2 \tau^2 \widehat{f}(\tau, \xi, \zeta) - c^2 i^2 \left( \xi^2 + \zeta^2 \right) \widehat{f}(\tau, \xi, \zeta) \right] \overline{\left[ i^2 \tau^2 \widehat{\vec{f}}(\tau, \xi, \zeta) - c^2 i^2 \left( \xi^2 + \zeta^2 \right) \widehat{\vec{f}}(\tau, \xi, \zeta) \right]} \dif \tau \dif \xi \dif \zeta
        \\ &= \iiint \left[ i^2 \left( \tau^2 - c^2 \left( \xi^2 + \zeta^2 \right) \right) \widehat{f}(\tau, \xi, \zeta) \right] \overline{\left[ i^2 \left( \tau^2 - c^2 \left( \xi^2 + \zeta^2 \right) \right) \widehat{\vec{f}}(\tau, \xi, \zeta) \right]} \dif \tau \dif \xi \dif \zeta
        \\ &= \iiint \left[ i^2 \left( \tau^2 - c^2 \left( \xi^2 + \zeta^2 \right) \right) \right] \overline{\left[ i^2 \left( \tau^2 - c^2 \left( \xi^2 + \zeta^2 \right) \right) \right]} \, \widehat{f}(\tau, \xi, \zeta) \overline{\widehat{\vec{f}}(\tau, \xi, \zeta)} \dif \tau \dif \xi \dif \zeta
        \\ &= \iiint \left( \tau^2 - c^2 \left( \xi^2 + \zeta^2 \right) \right) \overline{\left( \tau^2 - c^2 \left( \xi^2 + \zeta^2 \right) \right)} \, \widehat{f}(\tau, \xi, \zeta) \overline{\widehat{\vec{f}}(\tau, \xi, \zeta)} \dif \tau \dif \xi \dif \zeta;
        \\ &= \iiint \left[ \tau^2 - c^2 \left( \xi^2 + \zeta^2 \right) \right]^2 \, \widehat{f}(\tau, \xi, \zeta) \overline{\widehat{\vec{f}}(\tau, \xi, \zeta)} \dif \tau \dif \xi \dif \zeta;
    \end{align*}
    }
    For the initial condition terms, we will use the fact that
    \[ \widehat{u}(\xi, \zeta) = \frac{1}{\sqrt{2\pi}} \int \widehat{u}(\tau, \xi, \zeta) \dif \tau; \]
    It then follows:
    \begin{align*}
        &\iint \left[ f(0, x, y) - h(x, y) \right] \vec{f}(0, x, y) \dif x \dif y
        \\ &= \iint \left[ \frac{1}{\sqrt{2\pi}} \int \widehat{f}(\tau, \xi, \zeta) \dif\tau - \widehat{h}(\xi, \zeta) \right] \overline{\frac{1}{\sqrt{2\pi}} \int \widehat{\vec{f}}(\tau, \xi, \zeta) \dif \tau} \dif \xi \dif \zeta
        \\ &= \iiint \left[ \frac{1}{2\pi} \int \widehat{f}(\tau', \xi, \zeta) \dif\tau' - \frac{1}{\sqrt{2\pi}} \widehat{h}(\xi, \zeta) \right] \overline{\widehat{\vec{f}}(\tau, \xi, \zeta)} \dif \tau \dif \xi \dif \zeta,
    \end{align*}
    and
    \begin{align*}
        & \iint \frac{\partial f}{\partial t}(0, x, y) \, \frac{\partial \vec{f}}{\partial t}(0, x, y) \dif x \dif y
        \\ &= \iint \left( \frac{1}{\sqrt{2\pi}} \int i \tau \widehat{f}(\tau, \xi, \zeta) \dif \tau \right) \, \overline{\left( \frac{1}{\sqrt{2\pi}} \int i \tau \widehat{\vec{f}}(\tau, \xi, \zeta) \dif \tau \right)} \dif \xi \dif \zeta
        \\ &= \iiint \frac{\tau}{2\pi} \left( \int \tau' \widehat{f}(\tau', \xi, \zeta) \dif \tau' \right) \, \overline{\widehat{\vec{f}}(\tau, \xi, \zeta)} \dif \tau \dif \xi \dif \zeta.
    \end{align*}

    Putting it all together, we obtain:
    \begin{align*}
        & \langle \nabla L_\mathrm{PDE}(f), \vec{f} \rangle_{H^2(\R^3)}
        \\ &= \D L_\mathrm{PDE}(f; \vec{f})
        \\ &= \iiint \left[ \tau^2 - c^2 \left( \xi^2 + \zeta^2 \right) \right]^2 \, \widehat{f}(\tau, \xi, \zeta) \overline{\widehat{\vec{f}}(\tau, \xi, \zeta)} \dif \tau \dif \xi \dif \zeta
            \\ &\quad + \iiint \left[ \frac{1}{2\pi} \int \widehat{f}(\tau', \xi, \zeta) \dif\tau' - \frac{1}{\sqrt{2\pi}} \widehat{h}(\xi, \zeta) \right] \overline{\widehat{\vec{f}}(\tau, \xi, \zeta)} \dif \tau \dif \xi \dif \zeta,
            \\ &\quad + \iiint \frac{\tau}{2\pi} \left( \int \tau' \widehat{f}(\tau', \xi, \zeta) \dif \tau' \right) \, \overline{\widehat{\vec{f}}(\tau, \xi, \zeta)} \dif \tau \dif \xi \dif \zeta.
        \\ &= \iiint \left(
            \left[ \tau^2 - c^2 \left( \xi^2 + \zeta^2 \right) \right]^2 \, \widehat{f}(\tau, \xi, \zeta) \right.
            \\ &\quad \left. + \left[ \frac{1}{2\pi} \int \widehat{f}(\tau', \xi, \zeta) \dif\tau' - \frac{1}{\sqrt{2\pi}} \widehat{h}(\xi, \zeta) \right] \right.
            \\ &\quad \left. + \frac{\tau}{2\pi} \left( \int \tau' \widehat{f}(\tau', \xi, \zeta) \dif \tau' \right)
        \right) \overline{\widehat{\vec{f}}(\tau, \xi, \zeta)} \dif \tau \dif \xi \dif \zeta;
        \\ &= \iiint \left(
            \left[ \tau^2 - c^2 \left( \xi^2 + \zeta^2 \right) \right]^2 \, \widehat{f}(\tau, \xi, \zeta) \right.
            \\ &\quad \left. + \frac{1}{2\pi} \int \widehat{f}(\tau', \xi, \zeta) \dif\tau' - \frac{1}{\sqrt{2\pi}} \widehat{h}(\xi, \zeta) \right.
            \\ &\quad \left. + \frac{\tau}{2\pi} \left( \int \tau' \widehat{f}(\tau', \xi, \zeta) \dif \tau' \right)
        \right) \overline{\widehat{\vec{f}}(\tau, \xi, \zeta)} \dif \tau \dif \xi \dif \zeta;
    \end{align*}
    therefore, we have
    {\small
    \begin{align*}
        & [ \widehat{\nabla L_\mathrm{PDE}(f)} ](\tau, \xi, \zeta)
        \\ &= \frac{
            \left[ \tau^2 - c^2 \left( \xi^2 + \zeta^2 \right) \right]^2 \, \widehat{f}(\tau, \xi, \zeta)
            + \frac{1}{2\pi} \int \widehat{f}(\tau', \xi, \zeta) \dif\tau' - \frac{1}{\sqrt{2\pi}} \widehat{h}(\xi, \zeta)
            + \frac{\tau}{2\pi} \left( \int \tau' \widehat{f}(\tau', \xi, \zeta) \dif \tau' \right)
        }{1 + (\tau^2 + \xi^2 + \zeta^2) + (\tau^2 + \xi^2 + \zeta^2)^2}.
    \end{align*}
    }
\end{proof}

\subsubsection{Learning radiance fields}\label{suppl:sec:functional-gradient-3d}

Our scene is represented as a pair of functions
\begin{align*}
    \sigma : \R^3 \to \R
    \qquad\text{and}\qquad
    c : \R^3 \times \mathbb{S}^2 \to \R^3,
\end{align*}
corresponding to the density and color components, respectively; note that the color (returned in RGB) is taken to be view-dependent, while the density is view-independent.

With this functional representation of the scene, an infinitesimal ray $\gamma : [t_\near, t_\far] \to \R^3$ with direction $\gamma_\omega := \frac{\gamma(t_\far) - \gamma(t_\near)}{\norm{\gamma(t_\far) - \gamma(t_\near)}}$ can be rendered through
\begin{align*}
    \render_{\sigma,c}(\gamma) = \int_{t_\near}^{t_\far} T_\sigma(t; \gamma) \sigma\big(\gamma(t)\big) c\big(\gamma(t), \gamma_\omega\big) \dif t + T_\sigma(t_\far; \gamma) b,
\end{align*}
where $b \in \R^3$ is a fixed background color and $T_\sigma(\cdot)$ is the transmittance, defined as
\[ T_\sigma(t; \gamma) := \exp\left( -\int_{t_\near}^t \sigma\big(\gamma(s)\big) \dif s \right). \]
To render a full image from a scene, it will furthermore be useful to refer to an \emph{antialiased} rendering operation. Given a position in the image plane $p := (u, v) \in \R^2$, with corresponding infinitesimal rays given by $\gamma_p = \gamma_{u,v}$, we render by convolution with an antialiasing kernel $\widetilde{k} : \R^2 \to \R$ as
\begin{align*}
    \widetilde{\render}_{\sigma,c}(p) = \widetilde{\render}_{\sigma,c}(u, v)
    = \iint \render_{\sigma,c}(\gamma_{u',v'}) \widetilde{k}(u' - u, v' - v) \dif u' \dif v';
\end{align*}
Ideally $\widetilde{k}$ would be taken to be a sinc kernel so as to compensate for the sampling of the signal, though in practice it would be a more computable choice such as a bicubic filter.
Note that this can also be seen as a continuous version of supersampling antialiasing.

With these in hand, our goal is to solve a functional optimization problem
\begin{align}
    \argmin_{\sigma(\cdot),\,c(\cdot, \cdot)} L(\sigma, c) \qquad\text{with}\qquad L(\sigma, c) := \frac{1}{N} \sum_{i=1}^n \frac{1}{P} \sum_{j=1}^P \frac{1}{2} \norm{ \widetilde{\render}_{\sigma,c}(p_{i,j}) - Y_{i,j} }^2,
\end{align}
over $N$ training cameras\&images, each with $P$ pixels; here $p_{i,j}$ corresponds to the position in the pixel plane, and $Y_{i,j} \in \R^3$ corresponds to the color of the $j$-th pixel of the $i$-th image.
A position in pixel plane $p = (u,v)$ is mapped to the ray
\[ \gamma(t) = o + t A (u/f, -v/f, -1), \]
where $A \in \R^{3 \times 3}$ is the camera transformation matrix, $f$ is the focal length and $o \in \R^3$ is the camera position.

To treat this as an optimization problem in function space, we must first consider the domain of the loss functional $L(\cdot, \cdot)$, so as to ensure that functional gradients are well-defined.
To this end, let us first assume that the 3D domain of the density and color functions is contained to a $\Omega \subset \R^3$ that is compact (i.e., bounded and closed), endowed with the Lebesgue measure over $\Omega$.
We then take the density function to reside in standard $L^2$ space over $\Omega$, denoted $L^2(\Omega \to \R)$;
the color function, on the other hand, has to be taken to be on a ``mixture'' of being $L^2$ in the position $x \in \Omega \subset \R^3$ but following a more tame RKHS structure in the direction $\omega \in S^2$; this can be done by writing the color in curried form
\[ c : \Omega \to [\mathcal{H}^\mathbb{S}]^3, \]
where $\mathcal{H}^\mathbb{S}$ is a reproducing kernel Hilbert space (RKHS) of functions from the unit sphere $S^2$ to $\R$.
This is necessary so as to make the loss functional $L$ well-defined over a Hilbert space of functions.
The RKHS $\mathcal{H}^\mathbb{S}$ comes endowed with a reproducing kernel $K_{\mathcal{H}^\mathbb{S}} : \mathcal{H}^\mathbb{S} \times \mathcal{H}^\mathbb{S} \to \R$, defined by the reproducing property that $\inner{ h, K_{\mathcal{H}^\mathbb{S}}(\omega, \cdot) } = \inner{ h, K_{\mathcal{H}^\mathbb{S}}(\cdot, \omega) } = h(\omega)$ for any $h \in \mathcal{H}^\mathbb{S}$ and $\omega \in S^2$.

\begin{proposition}\label{thm:gradient-derivation}
    Suppose the antialiasing kernel $\widetilde{k}(\cdot, \cdot)$ has a non-empty support, and define the antialiased color differences
    \[ \widetilde{D}_{\sigma,c}(x; i,j) := \sum_{j=1}^P ( \widetilde{\render}_{\sigma,c}(p_{i,j}) - Y_{i,j} ) \, \widetilde{k}(p(x;i) - p_{i,j}). \]
    Then the loss functional $L : L^2(\Omega \to \R) \times L^2(\Omega \to [\mathcal{H}^\mathbb{S}]^3) \to \R$ is Fréchet differentiable, with gradients given by
    {\small
    \begin{align*}
        \nabla_\sigma L(\sigma, c)(x) &= \frac{1}{N} \sum_{i=1}^N a_i \sum_{k=1}^3 [\widetilde{D}_{\sigma,c}(x; i,j)]_k \biggl(
            T_\sigma(x; i) [c(x, \omega(x;i))]_k
            - \bigl( \render_{\sigma,c}(x; i) - \render^{[t_\near, t]}_{\sigma,c}(x; i) \bigr)
        \biggr),
        \\
        \nabla_{c_k} L(\sigma, c)(x, \omega) &= \frac{1}{N} \sum_{i=1}^N a_i [\widetilde{D}_{\sigma,c}(x; i,j)]_k T_\sigma(x; i) \sigma(x) K_{\mathcal{H}^\mathbb{S}}(\omega(x; i), \omega),
    \end{align*}
    }
    for constants $a_i := f_i^2 \ind[t(x;i) \in F] / P t(x;i)^2 \lvert\det A_i\rvert$.
\end{proposition}

Note that when the antialiasing kernel $\widetilde{k}$ has compact support, $\widetilde{D}_{\sigma,c}(x; i,j)$ can be efficiently computed by considering just a neighborhood around the $j$-th pixel.

\begin{proof}[Proof of Proposition~\ref{thm:gradient-derivation}]
    Let us start by deriving the directional (Gâteaux) derivative.
    It is handy to rewrite it as a scalar derivative evaluated at zero:
    \[
        \D L(\sigma, c; \vec{\sigma}, \vec{c})
        = \lim_{\delta \searrow 0} \frac{L(\sigma + \delta \vec{\sigma}, c + \delta \vec{c}) - L(\sigma, c)}{\delta}
        = \left[ \frac{\dif}{\dif \delta} L(\sigma + \delta \vec{\sigma}, c + \delta \vec{c}) \right]_{\delta=0};
    \]
    It then follows by straightforward computation:
    \begin{align*}
        & \left[ \frac{\dif}{\dif \delta} L(\sigma + \delta \vec{\sigma}, c + \delta \vec{c}) \right]_{\delta=0}
        \\ &= \left[ \frac{\dif}{\dif \delta} \frac{1}{N} \sum_{i=1}^N \frac{1}{P} \sum_{j=1}^P \frac{1}{2} \norm{ \widetilde{\render}_{\sigma + \delta \vec{\sigma}, c + \delta \vec{c}}(p_{i,j}) - Y_{i,j} }^2 \right]_{\delta=0}
        \\ &= \frac{1}{N} \sum_{i=1}^N \frac{1}{P} \sum_{j=1}^P \left[ \frac{\dif}{\dif \delta} \frac{1}{2} \norm{ \widetilde{\render}_{\sigma + \delta \vec{\sigma}, c + \delta \vec{c}}(p_{i,j}) - Y_{i,j} }^2 \right]_{\delta=0}
        \\ &= \frac{1}{N} \sum_{i=1}^N \frac{1}{P} \sum_{j=1}^P \sum_{k=1}^3 ( \widetilde{\render}_{\sigma,c}(p_{i,j}) - Y_{i,j} )_k \left[ \frac{\dif}{\dif \delta} [\widetilde{\render}_{\sigma + \delta \vec{\sigma}, c + \delta \vec{c}}(p_{i,j})]_k \right]_{\delta=0}
        \\ &= \frac{1}{N} \sum_{i=1}^N \frac{1}{P} \sum_{j=1}^P \sum_{k=1}^3 ( \widetilde{\render}_{\sigma,c}(p_{i,j}) - Y_{i,j} )_k \iint \left[ \frac{\dif}{\dif \delta} [\render_{\sigma+\delta \vec{\sigma},c+\delta \vec{c}}(\gamma_{u',v'})]_k \right]_{\delta=0} \widetilde{k}(u' - u, v' - v) \dif u' \dif v'.
    \end{align*}
    Now, the remaining derivative is as follows. For convenience, we let $\gamma \equiv \gamma_{u',v'}$.
    {\small
    \begin{align*}
        & \left[ \frac{\dif}{\dif \delta} [\render_{\sigma + \delta \vec{\sigma}, c + \delta \vec{c}}(\gamma)]_k \right]_{\delta=0}
        \\ &= \left[
            \frac{\dif}{\dif \delta} \int_{t_\near}^{t_\far}
                T_{\sigma + \delta \vec{\sigma}}(t;\gamma)
                \bigl(\sigma(\gamma(t)) + \delta \vec{\sigma}(\gamma(t))\bigr)
                \inner{ [c(\gamma(t))]_k + \delta [\vec{c}(\gamma(t))]_k, K_{\mathcal{H}^\mathbb{S}}(\gamma_\omega, \cdot) }
                \dif t
            \right.
            \\ &\qquad \left. + \frac{\dif}{\dif \delta} b_k T_{\sigma + \delta \vec{\sigma}}(t_\far;\gamma)
        \right]_{\delta=0}
        \\ &= \left[
            \frac{\dif}{\dif \delta} \int_{t_\near}^{t_\far}
                T_{\sigma}(t;\gamma) T_{\vec{\sigma}}(t;\gamma)^\delta
                \bigl(\sigma(\gamma(t)) + \delta \vec{\sigma}(\gamma(t))\bigr)
                \inner{ [c(\gamma(t))]_k + \delta [\vec{c}(\gamma(t))]_k, K_{\mathcal{H}^\mathbb{S}}(\gamma_\omega, \cdot) }
                \dif t \right.
            \\ &\qquad \left. + \frac{\dif}{\dif \delta} b_k T_{\sigma}(t_\far;\gamma) T_{\vec{\sigma}}(t_\far;\gamma)^\delta
        \right]_{\delta=0}
        \\ &=
            \int_{t_\near}^{t_\far}
                T_{\sigma}(t;\gamma)
                \left[ \frac{\dif}{\dif \delta}
                    T_{\vec{\sigma}}(t;\gamma)^\delta
                    (\sigma(\gamma(t)) + \delta \vec{\sigma}(\gamma(t)))
                    \inner{ [c(\gamma(t))]_k + \delta [\vec{c}(\gamma(t))]_k, K_{\mathcal{H}^\mathbb{S}}(\gamma_\omega, \cdot) }
                \right]_{\delta=0}
                \dif t
            \\ &\qquad + b_k T_{\sigma}(t_\far;\gamma) \left[ \frac{\dif}{\dif \delta} T_{\vec{\sigma}}(t_\far;\gamma)^\delta \right]_{\delta=0}
        \\ &=
            \int_{t_\near}^{t_\far}
                T_{\sigma}(t; \gamma)
                \biggl[
                    \bigl(\sigma(\gamma(t)) + \delta \vec{\sigma}(\gamma(t))\bigr)
                    \inner{ [c(\gamma(t))]_k + \delta [\vec{c}(\gamma(t))]_k, K_{\mathcal{H}^\mathbb{S}}(\gamma_\omega, \cdot) }
                    \log T_{\vec{\sigma}}(t;\gamma)
                    \\ &\qquad\qquad\qquad\quad + \frac{\dif}{\dif \delta}
                    \bigl(\sigma(\gamma(t)) + \delta \vec{\sigma}(\gamma(t))\bigr)
                    \inner{ [c(\gamma(t))]_k + \delta [\vec{c}(\gamma(t))]_k, K_{\mathcal{H}^\mathbb{S}}(\gamma_\omega, \cdot) }
                \biggr]_{\delta=0}
                \dif t
            \\ &\qquad + b_k T_{\sigma}(t_\far;\gamma) \log T_{\vec{\sigma}}(t_\far;\gamma)
        \\ &=
            \int_{t_\near}^{t_\far}
                T_{\sigma}(t;\gamma)
                \biggl[
                    \sigma(\gamma(t))
                    \langle [c(\gamma(t))]_k, K_{\mathcal{H}^\mathbb{S}}(\gamma_\omega, \cdot) \rangle
                    \log T_{\sigma}(t;\gamma)
                    \\ &\qquad\qquad\quad + \frac{\dif}{\dif \delta}
                    (\sigma(\gamma(t)) + \delta \vec{\sigma}(\gamma(t)))
                    \langle [c(\gamma(t))]_k + \delta [\vec{c}(\gamma(t))]_k, K_{\mathcal{H}^\mathbb{S}}(\gamma_\omega, \cdot) \rangle
                \biggr]_{\delta=0}
                \dif t
            \\ &\qquad + b_k T_{\sigma}(t_\far;\gamma) \log T_{\vec{\sigma}}(t_\far;\gamma)
        \\ &=
            \int_{t_\near}^{t_\far}
                T_{\sigma}(t;\gamma)
                \biggl[
                    \sigma(\gamma(t))
                    \langle [c(\gamma(t))]_k, K_{\mathcal{H}^\mathbb{S}}(\gamma_\omega, \cdot) \rangle
                    \log T_{\vec{\sigma}}(t;\gamma)
                    \\ &\qquad\qquad\qquad\qquad
                    + \bigl(\sigma(\gamma(t)) + \delta \vec{\sigma}(\gamma(t))\bigr) \langle [\vec{c}(\gamma(t))]_k, K_{\mathcal{H}^\mathbb{S}}(\gamma_\omega, \cdot) \rangle
                    \\ &\qquad\qquad\qquad\qquad
                    + \langle [c(\gamma(t))]_k + \delta [\vec{c}(\gamma(t))]_k, K_{\mathcal{H}^\mathbb{S}}(\gamma_\omega, \cdot) \rangle \vec{\sigma}(\gamma(t))
                \biggr]_{\delta=0}
                \dif t
            \\ &\qquad + b_k T_{\sigma}(t_\far;\gamma) \log T_{\vec{\sigma}}(t_\far;\gamma)
        \\ &=
            \int_{t_\near}^{t_\far}
                T_{\sigma}(t;\gamma)
                \biggl(
                    \sigma(\gamma(t))
                    \langle [c(\gamma(t))]_k, K_{\mathcal{H}^\mathbb{S}}(\gamma_\omega, \cdot) \rangle
                    \log T_{\vec{\sigma}}(t;\gamma)
                    \\ &\qquad\qquad\qquad\qquad
                    + \sigma(\gamma(t)) \langle [\vec{c}(\gamma(t))]_k, K_{\mathcal{H}^\mathbb{S}}(\gamma_\omega, \cdot) \rangle
                    \\ &\qquad\qquad\qquad\qquad
                    + \langle [c(\gamma(t))]_k, K_{\mathcal{H}^\mathbb{S}}(\gamma_\omega, \cdot) \rangle \vec{\sigma}(\gamma(t))
                \biggr)
                \dif t
            + b_k T_{\sigma}(t_\far;\gamma) \log T_{\vec{\sigma}}(t_\far;\gamma)
        \\ &=
            \int_{t_\near}^{t_\far}
                T_{\sigma}(t;\gamma)
                \biggl(
                    -\sigma(\gamma(t))
                    \langle [c(\gamma(t))]_k, K_{\mathcal{H}^\mathbb{S}}(\gamma_\omega, \cdot) \rangle
                    \int_{t_\near}^t \vec{\sigma}(\gamma(s)) \dif s
                    \\ &\qquad\qquad\qquad\qquad
                    + \sigma(\gamma(t)) \langle [\vec{c}(\gamma(t))]_k, K_{\mathcal{H}^\mathbb{S}}(\gamma_\omega, \cdot) \rangle
                    \\ &\qquad\qquad\qquad\qquad
                    + \langle [c(\gamma(t))]_k, K_{\mathcal{H}^\mathbb{S}}(\gamma_\omega, \cdot) \rangle \vec{\sigma}(\gamma(t))
                \biggr)
                \dif t
            - b_k T_{\sigma}(t_\far;\gamma) \int_{t_\near}^{t_\far} \vec{\sigma}(\gamma(s)) \dif s.
        \\ &=
            \int_{t_\near}^{t_\far}
                T_{\sigma}(t;\gamma)
                \langle [c(\gamma(t))]_k, K_{\mathcal{H}^\mathbb{S}}(\gamma_\omega, \cdot) \rangle \vec{\sigma}(\gamma(t))
            \dif t
            \\ &\qquad - \int_{t_\near}^{t_\far}
                T_{\sigma}(t;\gamma)
                \sigma(\gamma(t))
                \langle [c(\gamma(t))]_k, K_{\mathcal{H}^\mathbb{S}}(\gamma_\omega, \cdot) \rangle
                \int_{t_\near}^t \vec{\sigma}(\gamma(s)) \dif s
            \dif t
            \\ &\qquad - b_k T_{\sigma}(t_\far;\gamma) \int_{t_\near}^{t_\far} \vec{\sigma}(\gamma(s)) \dif s
            + \int_{t_\near}^{t_\far}
                T_{\sigma}(t;\gamma)
                \sigma(\gamma(t)) \langle [\vec{c}(\gamma(t))]_k, K_{\mathcal{H}^\mathbb{S}}(\gamma_\omega, \cdot) \rangle
            \dif t
    \end{align*}
    }
    Now, a slightly tricky step: in the second term, let's swap the integrals.
    {\small
    \begin{align*}
        &
            \int_{t_\near}^{t_\far}
                T_{\sigma}(t;\gamma)
                \langle [c(\gamma(t))]_k, K_{\mathcal{H}^\mathbb{S}}(\gamma_\omega, \cdot) \rangle \vec{\sigma}(\gamma(t))
            \dif t
            \\ &\qquad - \int_{t_\near}^{t_\far}
                T_{\sigma}(t;\gamma)
                \sigma(\gamma(t))
                \langle [c(\gamma(t))]_k, K_{\mathcal{H}^\mathbb{S}}(\gamma_\omega, \cdot) \rangle
                \int_{t_\near}^t \vec{\sigma}(\gamma(s)) \dif s
            \dif t
            \\ &\qquad - b_k T_{\sigma}(t_\far;\gamma) \int_{t_\near}^{t_\far} \vec{\sigma}(\gamma(s)) \dif s
            + \int_{t_\near}^{t_\far}
                T_{\sigma}(t;\gamma)
                \sigma(\gamma(t)) \langle [\vec{c}(\gamma(t))]_k, K_{\mathcal{H}^\mathbb{S}}(\gamma_\omega, \cdot) \rangle
            \dif t
        \\ &=
            \int_{t_\near}^{t_\far}
                T_{\sigma}(t;\gamma)
                \langle [c(\gamma(t))]_k, K_{\mathcal{H}^\mathbb{S}}(\gamma_\omega, \cdot) \rangle \vec{\sigma}(\gamma(t))
            \dif t
            \\ &\qquad - \int_{t_\near}^{t_\far} \int_{t_\near}^t
                T_{\sigma}(t;\gamma)
                \sigma(\gamma(t))
                \langle [c(\gamma(t))]_k, K_{\mathcal{H}^\mathbb{S}}(\gamma_\omega, \cdot) \rangle
                \vec{\sigma}(\gamma(s))
            \dif s \dif t
            \\ &\qquad - b_k T_{\sigma}(t_\far;\gamma) \int_{t_\near}^{t_\far} \vec{\sigma}(\gamma(s)) \dif s
            + \int_{t_\near}^{t_\far}
                T_{\sigma}(t;\gamma)
                \sigma(\gamma(t)) \langle [\vec{c}(\gamma(t))]_k, K_{\mathcal{H}^\mathbb{S}}(\gamma_\omega, \cdot) \rangle
            \dif t
        \\ &=
            \int_{t_\near}^{t_\far}
                T_{\sigma}(t;\gamma)
                \langle [c(\gamma(t))]_k, K_{\mathcal{H}^\mathbb{S}}(\gamma_\omega, \cdot) \rangle \vec{\sigma}(\gamma(t))
            \dif t
            \\ &\qquad - \int_{t_\near}^{t_\far} \int_{s}^{t_\far}
                T_{\sigma}(t;\gamma)
                \sigma(\gamma(t))
                \langle [c(\gamma(t))]_k, K_{\mathcal{H}^\mathbb{S}}(\gamma_\omega, \cdot) \rangle
                \vec{\sigma}(\gamma(s))
            \dif t \dif s
            \\ &\qquad - b_k T_{\sigma}(t_\far;\gamma) \int_{t_\near}^{t_\far} \vec{\sigma}(\gamma(s)) \dif s
            + \int_{t_\near}^{t_\far}
                T_{\sigma}(t;\gamma)
                \sigma(\gamma(t)) \langle [\vec{c}(\gamma(t))]_k, K_{\mathcal{H}^\mathbb{S}}(\gamma_\omega, \cdot) \rangle
            \dif t
        \\ &=
            \int_{t_\near}^{t_\far}
                T_{\sigma}(t;\gamma)
                \langle [c(\gamma(t))]_k, K_{\mathcal{H}^\mathbb{S}}(\gamma_\omega, \cdot) \rangle \vec{\sigma}(\gamma(t))
            \dif t
            \\ &\qquad - \int_{t_\near}^{t_\far} \int_{t}^{t_\far}
                T_{\sigma}(s;\gamma)
                \sigma(\gamma(s))
                \langle [c(\gamma(s))]_k, K_{\mathcal{H}^\mathbb{S}}(\gamma_\omega, \cdot) \rangle
                \vec{\sigma}(\gamma(t))
            \dif s \dif t
            \\ &\qquad - b_k T_{\sigma}(t_\far;\gamma) \int_{t_\near}^{t_\far} \vec{\sigma}(\gamma(s)) \dif s
            + \int_{t_\near}^{t_\far}
                T_{\sigma}(t;\gamma)
                \sigma(\gamma(t)) \langle [\vec{c}(\gamma(t))]_k, K_{\mathcal{H}^\mathbb{S}}(\gamma_\omega, \cdot) \rangle
            \dif t.
    \end{align*}
    }
    And now, we can turn this all into an $L^2$ inner product.
    {\small
    \begin{align*}
        &
            \int_{t_\near}^{t_\far}
                T_{\sigma}(t;\gamma)
                \langle [c(\gamma(t))]_k, K_{\mathcal{H}^\mathbb{S}}(\gamma_\omega, \cdot) \rangle \vec{\sigma}(\gamma(t))
            \dif t
            \\ &\qquad - \int_{t_\near}^{t_\far} \int_{t}^{t_\far}
                T_{\sigma}(s;\gamma)
                \sigma(\gamma(s))
                \langle [c(\gamma(s))]_k, K_{\mathcal{H}^\mathbb{S}}(\gamma_\omega, \cdot) \rangle
                \vec{\sigma}(\gamma(t))
            \dif s \dif t
            \\ &\qquad - b_k T_{\sigma}(t_\far;\gamma) \int_{t_\near}^{t_\far} \vec{\sigma}(\gamma(s)) \dif s
            + \int_{t_\near}^{t_\far}
                T_{\sigma}(t;\gamma)
                \sigma(\gamma(t)) \langle [\vec{c}(\gamma(t))]_k, K_{\mathcal{H}^\mathbb{S}}(\gamma_\omega, \cdot) \rangle
            \dif t.
        \\ &=
            \int_{t_\near}^{t_\far} \Biggl(
                T_{\sigma}(t;\gamma)
                \langle [c(\gamma(t))]_k, K_{\mathcal{H}^\mathbb{S}}(\gamma_\omega, \cdot) \rangle \vec{\sigma}(\gamma(t))
                \\ &\qquad\qquad\quad - \int_{t}^{t_\far} T_{\sigma}(s;\gamma)
                \sigma(\gamma(s))
                \langle [c(\gamma(s))]_k, K_{\mathcal{H}^\mathbb{S}}(\gamma_\omega, \cdot) \rangle
                \vec{\sigma}(\gamma(t))
                \dif s
                \\ &\qquad\qquad\quad - b_k T_{\sigma}(t_\far;\gamma) \vec{\sigma}(\gamma(t))
            \Biggr) \dif t
            \\ &\qquad + \int_{t_\near}^{t_\far}
                T_{\sigma}(t;\gamma)
                \sigma(\gamma(t)) \langle [\vec{c}(\gamma(t))]_k, K_{\mathcal{H}^\mathbb{S}}(\gamma_\omega, \cdot) \rangle
            \dif t
        \\ &=
            \int_{t_\near}^{t_\far} \Biggl(
                T_{\sigma}(t;\gamma)
                \langle [c(\gamma(t))]_k, K_{\mathcal{H}^\mathbb{S}}(\gamma_\omega, \cdot) \rangle
                \\ &\qquad\qquad\quad - \int_{t}^{t_\far} T_{\sigma}(s;\gamma)
                \sigma(\gamma(s))
                \langle [c(\gamma(s))]_k, K_{\mathcal{H}^\mathbb{S}}(\gamma_\omega, \cdot) \rangle
                \dif s
                - b_k T_{\sigma}(t_\far;\gamma)
            \Biggr) \vec{\sigma}(\gamma(t)) \dif t
            \\ &\qquad + \int_{t_\near}^{t_\far}
                T_{\sigma}(t;\gamma)
                \sigma(\gamma(t)) \langle [\vec{c}(\gamma(t))]_k, K_{\mathcal{H}^\mathbb{S}}(\gamma_\omega, \cdot) \rangle
            \dif t
        \\ &=
            \int_{t_\near}^{t_\far} \Biggl(
                T_{\sigma}(t;\gamma)
                \langle [c(\gamma(t))]_k, K_{\mathcal{H}^\mathbb{S}}(\gamma_\omega, \cdot) \rangle
                \\ &\qquad\qquad\quad - \int_{t}^{t_\far} T_{\sigma}(s;\gamma)
                \sigma(\gamma(s))
                \langle [c(\gamma(s))]_k, K_{\mathcal{H}^\mathbb{S}}(\gamma_\omega, \cdot) \rangle
                \dif s
                - b_k T_{\sigma}(t_\far;\gamma)
            \Biggr) \vec{\sigma}(\gamma(t)) \dif t
            \\ &\qquad + \int_{t_\near}^{t_\far}
                \langle [\vec{c}(\gamma(t))]_k, T_{\sigma}(t;\gamma) \sigma(\gamma(t)) K_{\mathcal{H}^\mathbb{S}}(\gamma_\omega, \cdot) \rangle
            \dif t
    \end{align*}
    }
    Plugging this back into the full directional derivative:
    {\small
    \begin{align*}
        & \D L(\sigma, c; \vec{\sigma}, \vec{c})
        \\ &= \frac{1}{N} \sum_{i=1}^N \frac{1}{P} \sum_{j=1}^P \sum_{k=1}^3 ( \widetilde{\render}_{\sigma,c}(p_{i,j}) - Y_{i,j} )_k \iint \left[ \frac{\dif}{\dif \delta} [\render_{\sigma+\delta \vec{\sigma},c+\delta \vec{c}}(\gamma_{u',v'})]_k \right]_{\delta=0} \widetilde{k}(u' - u, v' - v) \dif u' \dif v'
        \\ &= \frac{1}{N} \sum_{i=1}^N \frac{1}{P} \sum_{j=1}^P \sum_{k=1}^3 ( \widetilde{\render}_{\sigma,c}(p_{i,j}) - Y_{i,j} )_k \iint \Biggl(
            \\ &\qquad \int_{t_\near}^{t_\far}
                \langle [\vec{c}(\gamma_{u',v'}(t))]_k, T_{\sigma}(t;\gamma_{u',v'}) \sigma(\gamma_{u',v'}(t)) K_{\mathcal{H}^\mathbb{S}}((\gamma_{u',v'})_\omega, \cdot) \rangle
            \dif t
            \\ &\qquad
                + \int_{t_\near}^{t_\far} \biggl(
                    T_{\sigma}(t;\gamma_{u',v'})
                    \langle [c(\gamma_{u',v'}(t))]_k, K_{\mathcal{H}^\mathbb{S}}((\gamma_{u',v'})_\omega, \cdot) \rangle
                    \\ &\qquad - \int_{t}^{t_\far} T_{\sigma}(s;\gamma_{u',v'})
                    \sigma(\gamma_{u',v'}(s))
                    \langle [c(\gamma_{u',v'}(s))]_k, K_{\mathcal{H}^\mathbb{S}}((\gamma_{u',v'})_\omega, \cdot) \rangle
                    \dif s
                    \\ &\qquad - b_k T_{\sigma}(t_\far;\gamma_{u',v'})
                \biggr) \vec{\sigma}(\gamma_{u',v'}(t)) \dif t
            \Biggr) \widetilde{k}(u'-u, v'-v) \dif u' \dif v'
        \\ &= \frac{1}{N} \sum_{i=1}^N \frac{1}{P} \sum_{j=1}^P \sum_{k=1}^3 ( \widetilde{\render}_{\sigma,c}(p_{i,j}) - Y_{i,j} )_k \iint \biggl(
            \\ &\quad \int_{t_\near}^{t_\far}
                \langle [\vec{c}(\gamma_{u',v'}(t))]_k, T_{\sigma}(t;\gamma_{u',v'}) \sigma(\gamma_{u',v'}(t)) K_{\mathcal{H}^\mathbb{S}}((\gamma_{u',v'})_\omega, \cdot) \rangle
            \dif t \biggr) \widetilde{k}(u'-u,v'-v) \dif u' \dif v'
        \\ &\quad + \frac{1}{N} \sum_{i=1}^N \frac{1}{P} \sum_{j=1}^P \sum_{k=1}^3 ( \widetilde{\render}_{\sigma,c}(p_{i,j}) - Y_{i,j} )_k
            \\ &\qquad\qquad\qquad\qquad \cdot \iint \biggl(
            \int_{t_\near}^{t_\far} \Biggl(
                T_{\sigma}(t;\gamma_{u',v'})
                \langle [c(\gamma_{u',v'}(t))]_k, K_{\mathcal{H}^\mathbb{S}}((\gamma_{u',v'})_\omega, \cdot) \rangle
                \\ &\qquad\qquad - \int_{t}^{t_\far} T_{\sigma}(s;\gamma_{u',v'})
                \sigma(\gamma_{u',v'}(s))
                \langle [c(\gamma_{u',v'}(s))]_k, K_{\mathcal{H}^\mathbb{S}}((\gamma_{u',v'})_\omega, \cdot) \rangle
                \dif s
                \\ &\qquad\qquad - b_k T_{\sigma}(t_\far;\gamma_{u',v'})
            \Biggr) \vec{\sigma}(\gamma_{u',v'}(t)) \dif t
        \biggr) \widetilde{k}(u'-u,v'-v) \dif u' \dif v'.
        \\ &= \frac{1}{N} \sum_{i=1}^N \frac{1}{P} \sum_{j=1}^P \sum_{k=1}^3 \left[ ( \widetilde{\render}_{\sigma,c}(p_{i,j}) - Y_{i,j} )_k \right.
            \\ &\qquad \left. \cdot \iint \biggl(
            \int_{t_\near}^{t_\far}
                \langle [\vec{c}(\gamma_{u',v'}(t))]_k, T_{\sigma}(t;\gamma_{u',v'}) \sigma(\gamma_{u',v'}(t)) K_{\mathcal{H}^\mathbb{S}}((\gamma_{u',v'})_\omega, \cdot) \rangle
            \dif t \biggr) \widetilde{k}(u'-u,v'-v) \dif u' \dif v' \right]
        \\ &\quad + \frac{1}{N} \sum_{i=1}^N \frac{1}{P} \sum_{j=1}^P \sum_{k=1}^3 ( \widetilde{\render}_{\sigma,c}(p_{i,j}) - Y_{i,j} )_k \Biggl(
            \\ &\qquad \iint \biggl(
            \int_{t_\near}^{t_\far}
                T_{\sigma}(t;\gamma_{u',v'}) \vec{\sigma}(\gamma_{u',v'}(t))
                \langle [c(\gamma_{u',v'}(t))]_k, K_{\mathcal{H}^\mathbb{S}}((\gamma_{u',v'})_\omega, \cdot) \rangle \dif t \biggr) \widetilde{k}(u'-u,v'-v) \dif u' \dif v'
                \\ &\qquad - \iint \biggl( \int_{t_\near}^{t_\far} \biggl( \int_{t}^{t_\far} T_{\sigma}(s;\gamma_{u',v'})
                \sigma(\gamma_{u',v'}(s))
                \langle [c(\gamma_{u',v'}(s))]_k, K_{\mathcal{H}^\mathbb{S}}((\gamma_{u',v'})_\omega, \cdot) \rangle
                \dif s
                \\ &\qquad + b_k T_{\sigma}(t_\far;\gamma_{u',v'})
            \biggr) \vec{\sigma}(\gamma_{u',v'}(t)) \dif t
        \biggr) \widetilde{k}(u'-u,v'-v) \dif u' \dif v' \Biggr).
        \\ &= \frac{1}{N} \sum_{i=1}^N \frac{1}{P} \sum_{j=1}^P \sum_{k=1}^3 ( \widetilde{\render}_{\sigma,c}(p_{i,j}) - Y_{i,j} )_k
            \\ &\qquad \iint \biggl(
            \int_{t_\near}^{t_\far}
                \langle [\vec{c}(\gamma_{u',v'}(t))]_k, T_{\sigma}(t;\gamma_{u',v'}) \sigma(\gamma_{u',v'}(t)) K_{\mathcal{H}^\mathbb{S}}((\gamma_{u',v'})_\omega, \cdot) \rangle
            \dif t \biggr) \widetilde{k}(u'-u,v'-v) \dif u' \dif v'
        \\ &\quad + \frac{1}{N} \sum_{i=1}^N \frac{1}{P} \sum_{j=1}^P \sum_{k=1}^3 ( \widetilde{\render}_{\sigma,c}(p_{i,j}) - Y_{i,j} )_k \Biggl(
            \\ &\qquad \iint \biggl(
            \int_{t_\near}^{t_\far}
                T_{\sigma}(t;\gamma_{u',v'}) \vec{\sigma}(\gamma_{u',v'}(t))
                \langle [c(\gamma(t))]_k, K_{\mathcal{H}^\mathbb{S}}((\gamma_{u',v'})_\omega, \cdot) \rangle \dif t \biggr) \widetilde{k}(u'-u,v'-v) \dif u' \dif v'
            \\ &\qquad - \iint \biggl( \int_{t_\near}^{t_\far} \bigl( \render_{\sigma,c}(\gamma_{u',v'}) - \render_{\sigma,c}(\gamma_{u',v'}^{[t_\near, t]})
            \bigr) \vec{\sigma}(\gamma_{u',v'}(t)) \dif t
        \biggr) \widetilde{k}(u'-u,v'-v) \dif u' \dif v' \Biggr).
    \end{align*}
    }
    Now, to help transform this into an $L^2$ inner product, let us introduce a lemma.

    \begin{lemma}
        For the $i$-th camera and any $g : \R^3 \to \R$,
        {\small
        \begin{align*}
            \iint \int_{t_\near}^{t_\far} g(\gamma_{u',v'}(t)) \dif t \, \widetilde{k}(u'-u,v'-v) \dif u' \dif v'
            = \frac{f_i^2}{\lvert \det A_i \rvert} \int_{\R^3} \frac{\ind[t(x; i) \in F]}{t(x; i)^2} g(x) \widetilde{k}(p(x; i)-p_{i,j}) \dif x,
        \end{align*}
        }
        where $F = [t_\near, t_\far]$, and $A_i, f_i$ are camera parameters.
    \end{lemma}
    \begin{proof}[Proof (of the lemma)]
        First write $\gamma_{u',v'}(t) = o_i + t A_i q_{u',v'}$, for $q_{u',v'} = (u'/f_i, -v'/f_i, -1)$;
        so we have the triple integral
        \begin{align*}
            & \iint \int_{t_\near}^{t_\far} g(o_i + t A_i q_{u',v'}) \widetilde{k}(u'-u,v'-v) \dif t \dif u' \dif v'
            \\ &\quad = \iint \int_{t_\near}^{t_\far} g(o_i + t A_i q_{u',v'}) \widetilde{k}(p(o_i + t A_i q_{u',v'}; i)-p_{i,j}) \dif t \dif u' \dif v',
            \\ &\quad = \iint \int_{t_\near}^{t_\far} g(\Phi(u',v',t)) \widetilde{k}(p(\Phi(u',v',t); i)-p_{i,j}) \dif t \dif u' \dif v',
        \end{align*}
        where $p(\cdot; i)$ is such that $p(o + t A q_{u',v'}; i) = (u',v')$, and
        \[ \Phi\bigl(u, v, t\bigr) = o + t A q_{u,v} = o + t A (u/f, -v/f, -1). \]
        We will do a change-of-variables with $\Phi$.
        Note that it is invertible and differentiable, with Jacobian matrix as follows:
        \begin{align*}
            J\Phi\bigl(u, v, t\bigr) = A \begin{bmatrix}
                t/f & 0 & u/f
                \\
                0 & -t/f & -v/f
                \\
                0 & 0 & -1
            \end{bmatrix},
            \qquad
            \det\left(J\Phi\bigl(u, v, t\bigr)\right) = \det(A) \cdot \frac{t^2}{f^2}.
        \end{align*}
        Now, let $D = \R^2 \times [t_\near, t_\far]$.
        By integration by substitution,
        \begin{align*}
            & \iint \int_{t_\near}^{t_\far} g(\Phi(u',v',t)) \widetilde{k}(p(\Phi(u',v',t); i)-p_{i,j}) \dif t \dif u' \dif v'
            \\ &= \iiint_D g(\Phi(u',v',t)) \widetilde{k}(p(\Phi(u',v',t); i)-p_{i,j}) \dif t \dif u' \dif v'
            \\ &= \iiint_{\Phi^{-1}(\Phi(D))} g(\Phi(u',v',t)) \widetilde{k}(p(\Phi(u',v',t); i)-p_{i,j}) \dif t \dif u' \dif v'
            \\ &= \iiint_{\Phi(D)} g(\Phi(\Phi^{-1}(x))) \widetilde{k}(p(\Phi(\Phi^{-1}(x)); i)-p_{i,j}) \cdot \lvert \det J\Phi^{-1}(x) \rvert \dif x
            \\ &= \iiint_{\Phi(D)} g(x) \widetilde{k}(p(x; i)-p_{i,j}) \cdot \lvert \det J\Phi^{-1}(x) \rvert \dif x
        \end{align*}
        and by the inverse function theorem,
        \begin{align*}
            & \iiint_{\Phi(D)} g(x) \widetilde{k}(p(x; i)-p_{i,j}) \cdot \lvert \det J\Phi^{-1}(x) \rvert \dif x
            \\ &= \iiint_{\Phi(D)} g(x) \widetilde{k}(p(x; i)-p_{i,j}) \cdot \lvert \det J\Phi(\Phi^{-1}(x)) \rvert^{-1} \dif x
            \\ &= \iiint_{\Phi(D)} g(x) \widetilde{k}(p(x; i)-p_{i,j}) \cdot \lvert \det A \rvert^{-1} \frac{f^2}{\Phi^{-1}_t(x)^2} \dif x
            \\ &= \frac{f^2}{\lvert \det A \rvert} \iiint_{\Phi(D)} \frac{1}{\Phi^{-1}_t(x)^2} g(x) \widetilde{k}(p(x; i)-p_{i,j}) \dif x
            \\ &= \frac{f^2}{\lvert \det A \rvert} \iiint_{\Phi(D)} \frac{1}{t(x; i)^2} g(x) \widetilde{k}(p(x; i)-p_{i,j}) \dif x
            \\ &= \frac{f^2}{\lvert \det A \rvert} \int_{\R^3} \frac{\ind[t(x; i) \in F]}{t(x; i)^2} g(x) \widetilde{k}(p(x; i)-p_{i,j}) \dif x.
            \qedhere
        \end{align*}
    \end{proof}

    So, applying the lemma, we have that
    {\small
    \begin{align*}
        & \frac{1}{N} \sum_{i=1}^N \frac{1}{P} \sum_{j=1}^P \sum_{k=1}^3 \left[ ( \widetilde{\render}_{\sigma,c}(p_{i,j}) - Y_{i,j} )_k \right.
        \\ &\qquad \left. \iint \biggl(
            \int_{t_\near}^{t_\far}
                \langle [\vec{c}(\gamma_{u',v'}(t))]_k, T_{\sigma}(t;\gamma_{u',v'}) \sigma(\gamma_{u',v'}(t)) K_{\mathcal{H}^\mathbb{S}}((\gamma_{u',v'})_\omega, \cdot) \rangle
            \dif t \biggr) \widetilde{k}(u'-u,v'-v) \dif u' \dif v' \right]
        \\ &\quad + \frac{1}{N} \sum_{i=1}^N \frac{1}{P} \sum_{j=1}^P \sum_{k=1}^3 ( \widetilde{\render}_{\sigma,c}(p_{i,j}) - Y_{i,j} )_k \Biggl(
        \\ &\qquad \iint \biggl(
            \int_{t_\near}^{t_\far}
                T_{\sigma}(t;\gamma_{u',v'}) \vec{\sigma}(\gamma_{u',v'}(t))
                \langle [c(\gamma(t))]_k, K_{\mathcal{H}^\mathbb{S}}((\gamma_{u',v'})_\omega, \cdot) \rangle \dif t \biggr) \widetilde{k}(u'-u,v'-v) \dif u' \dif v'
            \\ &\qquad - \iint \biggl( \int_{t_\near}^{t_\far} \bigl( \render_{\sigma,c}(\gamma_{u',v'}) - \render_{\sigma,c}(\gamma_{u',v'}^{[t_\near, t]})
            \bigr) \vec{\sigma}(\gamma_{u',v'}(t)) \dif t
        \biggr) \widetilde{k}(u'-u,v'-v) \dif u' \dif v' \Biggr)
        \\ &= \frac{1}{N} \sum_{i=1}^N \frac{1}{P} \sum_{j=1}^P \sum_{k=1}^3 \left[ ( \widetilde{\render}_{\sigma,c}(p_{i,j}) - Y_{i,j} )_k \frac{f_i^2}{\lvert \det A_i \rvert} \right.
        \\ &\qquad \left. \int_{\R^3} \frac{\ind[t(x;i) \in F]}{t(x;i)^2}
                \langle [\vec{c}(x)]_k, T_{\sigma}(x;i) \sigma(x) K_{\mathcal{H}^\mathbb{S}}((\gamma_{u',v'})_\omega, \cdot) \rangle
            \widetilde{k}(p(x;i)-p_{i,j}) \dif x \right]
        \\ &\quad + \frac{1}{N} \sum_{i=1}^N \frac{1}{P} \sum_{j=1}^P \sum_{k=1}^3 ( \widetilde{\render}_{\sigma,c}(p_{i,j}) - Y_{i,j} )_k \Biggl(
        \\ &\qquad \frac{f_i^2}{\lvert \det A_i \rvert} \int_{\R^3} \frac{\ind[t(x;i) \in F]}{t(x;i)^2}
                T_{\sigma}(x;i) \vec{\sigma}(x)
                \langle [c(x)]_k, K_{\mathcal{H}^\mathbb{S}}((\gamma_{u',v'})_\omega, \cdot) \rangle \widetilde{k}(p(x;i)-p_{i,j}) \dif x
            \\ &\qquad - \frac{f_i^2}{\lvert \det A_i \rvert} \int_{\R^3} \frac{\ind[t(x;i) \in F]}{t(x;i)^2} \bigl( \render_{\sigma,c}(\gamma_{p(x;i)}) - \render_{\sigma,c}(\gamma_{p(x;i)}^{[t_\near, t]})
            \bigr) \vec{\sigma}(x)
        \widetilde{k}(p(x;i)-p_{i,j}) \dif x \Biggr)
        \\ &= \int_{\R^3} \frac{1}{N} \sum_{i=1}^N \frac{1}{P} \sum_{j=1}^P \sum_{k=1}^3 \left[ ( \widetilde{\render}_{\sigma,c}(p_{i,j}) - Y_{i,j} )_k \right.
        \\ &\qquad \left. \frac{f_i^2}{\lvert \det A_i \rvert} \frac{\ind[t(x;i) \in F]}{t(x;i)^2}
                \langle [\vec{c}(x)]_k, T_{\sigma}(x;i) \sigma(x) K_{\mathcal{H}^\mathbb{S}}((\gamma_{u',v'})_\omega, \cdot) \rangle
            \widetilde{k}(p(x;i)-p_{i,j}) \dif x \right]
        \\ &\quad + \int_{\R^3} \frac{1}{N} \sum_{i=1}^N \frac{1}{P} \sum_{j=1}^P \sum_{k=1}^3 ( \widetilde{\render}_{\sigma,c}(p_{i,j}) - Y_{i,j} )_k \Biggl(
        \\ &\qquad \frac{f_i^2}{\lvert \det A_i \rvert} \frac{\ind[t(x;i) \in F]}{t(x;i)^2}
                T_{\sigma}(x;i) \vec{\sigma}(x)
                \langle [c(x)]_k, K_{\mathcal{H}^\mathbb{S}}((\gamma_{u',v'})_\omega, \cdot) \rangle \widetilde{k}(p(x;i)-p_{i,j})
            \\ &\qquad - \frac{f_i^2}{\lvert \det A_i \rvert} \frac{\ind[t(x;i) \in F]}{t(x;i)^2} \bigl( \render_{\sigma,c}(\gamma_{p(x;i)}) - \render_{\sigma,c}(\gamma_{p(x;i)}^{[t_\near, t]})
            \bigr) \vec{\sigma}(x)
        \widetilde{k}(p(x;i)-p_{i,j}) \Biggr) \dif x
        \\ &= \int_{\R^3} \sum_{k=1}^3 \frac{1}{N} \sum_{i=1}^N \frac{f_i^2 \ind[t(x;i) \in F]}{P t(x;i)^2 \lvert \det A_i \rvert} \left[ \langle [\vec{c}(x)]_k, T_{\sigma}(x;i) \sigma(x) K_{\mathcal{H}^\mathbb{S}}((\gamma_{u',v'})_\omega, \cdot) \rangle \right.
            \\ &\qquad \left. \sum_{j=1}^P ( \widetilde{\render}_{\sigma,c}(p_{i,j}) - Y_{i,j} )_k \, \widetilde{k}(p(x;i)-p_{i,j}) \dif x \right]
        \\ &\quad + \int_{\R^3} \sum_{k=1}^3 \frac{1}{N} \sum_{i=1}^N \frac{f_i^2 \ind[t(x;i) \in F]}{P t(x;i)^2 \lvert \det A_i \rvert} \Biggl( T_{\sigma}(x;i) \vec{\sigma}(x)
                \langle [c(x)]_k, K_{\mathcal{H}^\mathbb{S}}((\gamma_{u',v'})_\omega, \cdot) \rangle
            \\ &\qquad - \bigl( \render_{\sigma,c}(\gamma_{p(x;i)}) - \render_{\sigma,c}(\gamma_{p(x;i)}^{[t_\near, t]}) \bigr) \vec{\sigma}(x)
        \Biggr) \sum_{j=1}^P ( \widetilde{\render}_{\sigma,c}(p_{i,j}) - Y_{i,j} )_k \widetilde{k}(p(x;i)-p_{i,j}) \dif x
        \\ &= \int_{\R^3} \sum_{k=1}^3 \left\langle [\vec{c}(x)]_k, \frac{1}{N} \sum_{i=1}^N \left[ \frac{f_i^2 \ind[t(x;i) \in F]}{P t(x;i)^2 \lvert \det A_i \rvert} T_{\sigma}(x;i) \sigma(x) K_{\mathcal{H}^\mathbb{S}}((\gamma_{u',v'})_\omega, \cdot) \right.\right.
            \\ &\qquad \left.\left. \sum_{j=1}^P ( \widetilde{\render}_{\sigma,c}(p_{i,j}) - Y_{i,j} )_k \, \widetilde{k}(p(x;i)-p_{i,j}) \right] \right\rangle \dif x
        \\ &\quad + \int_{\R^3} \vec{\sigma}(x) \sum_{k=1}^3 \frac{1}{N} \sum_{i=1}^N \frac{f_i^2 \ind[t(x;i) \in F]}{P t(x;i)^2 \lvert \det A_i \rvert} \Biggl( T_{\sigma}(x;i) [c(x, (\gamma_{u',v'})_\omega)]_k
            \\ &\qquad - \bigl( \render_{\sigma,c}(\gamma_{p(x;i)}) - \render_{\sigma,c}(\gamma_{p(x;i)}^{[t_\near, t]}) \bigr)
        \Biggr) \sum_{j=1}^P ( \widetilde{\render}_{\sigma,c}(p_{i,j}) - Y_{i,j} )_k \widetilde{k}(p(x;i)-p_{i,j}) \dif x.
    \end{align*}
    }
    And thus, the functional gradients are given by:
    {\small
    \begin{align*}
        \nabla_\sigma L(\sigma, c)(x) &= \sum_{k=1}^3 \frac{1}{N} \sum_{i=1}^N \frac{f_i^2 \ind[t(x;i) \in F]}{P t(x;i)^2 \lvert\det A_i\rvert} \biggl(
            T_\sigma(x; i) [c(x, \omega(x;i))]_k
            \\ &\qquad\qquad - \bigl( \render_{\sigma,c}(x; i) - \render^{[t_\near, t]}_{\sigma,c}(x; i) \bigr)
        \biggr) \sum_{j=1}^P ( \widetilde{\render}_{\sigma,c}(p_{i,j}) - Y_{i,j} )_k \, \widetilde{k}(p(x;i)-p_{i,j}),
        \\
        \nabla_{c_k} L(\sigma, c)(x, \omega) &= \frac{1}{N} \sum_{i=1}^N \left[ \frac{f_i^2 \ind[t(x;i) \in F]}{P t(x;i)^2 \lvert\det A_i\rvert} T_\sigma(x; i) \sigma(x) K_{\mathcal{H}^\mathbb{S}}(\omega(x; i), \omega) \right.
        \\ &\qquad\qquad\quad \left. \sum_{j=1}^P ( \widetilde{\render}_{\sigma,c}(p_{i,j}) - Y_{i,j} )_k \, \widetilde{k}(p(x;i)-p_{i,j}) \right],
    \end{align*}
    }
    and using $\widetilde{D}_{\sigma,c}(x; i,j) = \sum_{j=1}^P ( \widetilde{\render}_{\sigma,c}(p_{i,j}) - Y_{i,j} ) \, \widetilde{k}(p(x;i)-p_{i,j})$ and $a_i = (f_i^2 \ind[t(x;i) \in F])/(P t(x;i)^2 \lvert\det A_i\rvert)$ we obtain
    {\small
    \begin{align*}
        \nabla_\sigma L(\sigma, c)(x) &= \frac{1}{N} \sum_{i=1}^N a_i \sum_{k=1}^3 [\widetilde{D}_{\sigma,c}(x; i,j)]_k \biggl(
            T_\sigma(x; i) [c(x, \omega(x;i))]_k
            - \bigl( \render_{\sigma,c}(x; i) - \render^{[t_\near, t]}_{\sigma,c}(x; i) \bigr)
        \biggr),
        \\
        \nabla_{c_k} L(\sigma, c)(x, \omega) &= \frac{1}{N} \sum_{i=1}^N a_i [\widetilde{D}_{\sigma,c}(x; i,j)]_k T_\sigma(x; i) \sigma(x) K_{\mathcal{H}^\mathbb{S}}(\omega(x; i), \omega).
    \end{align*}
    }
    This establishes that $L$ is (linearly) Gâteaux differentiable.
    It is also immediately seen to be Fréchet differentiable, as the gradients are continuous with relation to $\sigma$ and $c$.
\end{proof}

\section{Experiment details}

\subsection{Regression in an RKHS}

For the experiments we use an RBF kernel $K(x, y) = \exp(-100 \lVert x - y \rVert^2)$.
We use a simple tree-based approximation scheme which naively partitions splits at the midpoints.
This is mainly for ease of (efficient) implementation in JAX, and in no way fundamental to our approach.
To upper bound the approximation error, we use the fact that our trees are piecewise constant over rectangular partitions of the space, and bound the sup-norm of each leaf via a Lipschitz bound on the gradient (using the fact that the RBF kernel is Lipschitz).
The neural network is a standard MLP with two hidden layers of size 256; we note that modifying this doesn't change the conclusions in any significant manner.

\subsection{Solving the wave equation}

The neural network baseline has a Fourier feature embedding layer~\citep{fourier-embedding}, which is essential for the neural network to even be capable of optimizing this loss; without this embedding, no matter how large (or small) the network or optimizer hyperparameters, it stays stuck and does not fit the solution at all.
As the loss has an integral in its definition, the neural network optimizes it on batches of 4096 points uniformly sampled over the domain; lowering this batch size leads to worse results.
For our method, the inverse Fourier transform can be directly computed as a sum of sincs, using the fact that a box function in frequency space corresponds to a sinc in the primary space.
Finally, we bound the approximation error via a combination of a Lipschitz argument (for the interior of the frequency-space grid) and a bound on the Sobolev symbol for the content outside of the grid.
For the reference solution in the figure, we solve the equation using a finite differences method with an extremely fine grid.

\subsection{Learning radiance fields}

We train our method on 24 $160 \times 160$ RGB images, and evaluate on 25 $160 \times 160$ RGB images of novel views on the test subset.
Both Approx. FGD (all resolutions) and Adaptive FGD used identical optimization hyperparameters.
They employed a learning rate of 8 for the color function $c$ and 20 for the density function $\sigma$, approximated the gradient by taking 8 samples per voxel and averaging them for the density, and fitting spherical harmonics coefficients through one step of Newton's method.
On the other hand, the NN used Adam optimizer with learning rate of $1e-4$ and $\beta = (0.9, 0.999)$.
Since the NN cannot perform full-batch gradient descent due to VRAM limitations, we used mini-batches of $80\times 80$ rays and counted $4 \times 24$ mini-batches as an optimization step (to match the full-batch updates of our method).
Note that, since the NN traces one ray per pixel, $4 \times 24$ mini-batches correspond exactly to the number of rays in the full training data.

\end{document}